\documentclass[11pt,british]{article}
\usepackage[T1]{fontenc}
\usepackage[latin9]{inputenc}
\usepackage{geometry}
\usepackage{enumitem}
\usepackage{amsmath}
\usepackage{amsthm}
\usepackage{amssymb}

\makeatletter
\theoremstyle{plain}
\newtheorem{thm}{\protect\theoremname}[section]
  \theoremstyle{definition}
  \newtheorem{defn}[thm]{\protect\definitionname}
  \theoremstyle{plain}
  \newtheorem{lem}[thm]{\protect\lemmaname}
  \theoremstyle{plain}
  \newtheorem{prop}[thm]{\protect\propositionname}
  \theoremstyle{remark}
  \newtheorem{rem}[thm]{\protect\remarkname}
  \theoremstyle{plain}
  \newtheorem{cor}[thm]{\protect\corollaryname}
  \theoremstyle{definition}
  \newtheorem{example}[thm]{\protect\examplename}

\newlength{\bibitemsep}
\newlength{\bibparskip}
\let\oldthebibliography\thebibliography
\renewcommand\thebibliography[1]{%
  \oldthebibliography{#1}%
  \setlength{\parskip}{\bibitemsep}%
  \setlength{\itemsep}{\bibparskip}%
}

\makeatother

\usepackage{babel}
  \providecommand{\corollaryname}{Corollary}
  \providecommand{\definitionname}{Definition}
  \providecommand{\examplename}{Example}
  \providecommand{\lemmaname}{Lemma}
  \providecommand{\propositionname}{Proposition}
  \providecommand{\remarkname}{Remark}
\providecommand{\theoremname}{Theorem}

\begin{document}

\title{Existence of Geometric Ergodic Periodic Measures of Stochastic Differential
Equations}

\author{Chunrong Feng, Huaizhong Zhao and Johnny Zhong\\Department of Mathematical
Sciences, Loughborough University, LE11 3TU, UK\\E-mail address:
C.Feng@lboro.ac.uk, H.Zhao@lboro.ac.uk, J.Zhong@lboro.ac.uk}

\maketitle
\renewcommand{\labelenumi}{(\roman{enumi})}

\numberwithin{equation}{section}
\begin{abstract}
Periodic measures are the time-periodic counterpart to invariant measures
for dynamical systems and can be used to characterise the long-term
periodic behaviour of stochastic systems. This paper gives sufficient
conditions for the existence, uniqueness and geometric convergence
of a periodic measure for time-periodic Markovian processes on a locally
compact metric space in great generality. In particular, we apply
these results in the context of time-periodic weakly dissipative stochastic
differential equations, gradient stochastic differential equations
as well as Langevin equations. We will establish the Fokker-Planck
equation that the density of the periodic measure sufficiently and
necessarily satisfies. Applications to physical problems shall be
discussed with specific examples. 

\end{abstract}

\section{Introduction}

In existing literature, there are a vast number of results concerning
asymptotic behaviour of both deterministic and stochastic autonomous
systems. In particular, there are many powerful results on the existence
and uniqueness of a limiting invariant measure of time-homogeneous
Markovian systems of both finite and infinite dimensions (\cite{DaPratoZabczyk_ErgodicityInfDimensions,MeynTweedie1_Chains,NorrisMC,MattinglyStuartHigham_ErgodicityForSDEs,Hasminskii,Hairer_CouplingSPDEs,MeynTweedie_MarkovChainStability,Mattingly_Progress_SNS,MeynTweedie3_Processes}).
While limiting invariant measure captures the idea that the system
``settles'' towards an equilibrium, it does not accommodate for
systems that are asymptotically periodic. Needless to say, the (asymptotic)
periodic solution is natural to study and a central branch within
the theory of dynamical systems. However, due to the delicate nature
of combining periodicity and randomness, there is still a gap in literature
for asymptotic random periodic behaviour of stochastic systems. Filling
this gap, in \cite{CFHZ2016}, the authors defined rigorously periodic
measures which play the role as the time-periodic counterpart of invariant
measures. In particular, periodic measures can characterise asymptotic
periodic behaviour for stochastic systems. See also \cite{FengLiuZhaoNumericalRPP,FengWuZhaoRPS}
for some discussions.

As with invariant measures, periodic measures have both theoretical
and practical applications to physical sciences. In this paper, we
establish explicit criteria for the uniqueness and geometric convergence
of a periodic measure for time-periodic Markovian processes applicable
to a great general setting. Our results apply to ``periodically forced''
stochastic systems which have a range of applications. We refer readers
to \cite{JungPerioidcallyDrivenSys,ZhouMossJung} and references therein
for examples from biology and physics. A notable example includes
the overdamped Duffing Oscillator which has been used to model climate
dynamics \cite{NicolisPeriodicForcing,BenziStochRes} to portray the
physical phenomena of stochastic resonance. The stochastic resonance
model introduced in \cite{BenziStochRes} offered a reasonable physical
explanation about the peak observed in the power spectrum of paleoclimatic
variations in the last 700,000 years at a periodicity of around $10^{5}$
years. This is in complementary with smaller peaks at periods of $2\times10^{4}$
and $4\times10^{2}$ years. The major peak represents dramatic climate
change to a temperature change of 10K in Kelvin scale. Except for
the dramatic changes, temperature seems to oscillate around fixed
values. This phenomenon was suggested to be related to variations
in the earth's orbital parameter which also has a similar periodic
pattern of changes \cite{Mil30,HIS76}. Such studies were able to
reproduce smaller peaks, but failed to explain the $10^{5}$-year
cycle major peak. In physics literature, the theory of stochastic
resonance for stochastic periodically forced double well potential
provides a mathematical model of the transitions between the two equilibria
interpreted as climates of the ice age and interglacial period respectively
\cite{BenziStochRes,NicolisPeriodicForcing}. Periodic forcing corresponds
to the annual mean variation in insolation due to changes in ellipticity
of the earth's orbit, while noise stimulates the global effect of
relatively short-term fluctuations in the atmospheric and oceanic
circulations on the long-term temperature behaviour. The transition
is driven by the noise in the system and happens more likely when
one of the well is at or near the highest position and the other one
is at the lowest position due to the periodicity. This striking phenomenon
is intuitively correct and agrees with the reality. It is noted that
stochastic resonance occurs for the right set of parameters in the
stochastic periodic double well model, suggested by numerical simulations
\cite{GHJM98,MW89,ANMS99,CherubiniStochRes}. The concept of periodic
measures and ergodicity (\cite{CFHZ2016}) provides a rigorous framework
and new insight for understanding such physical phenomena. 

We will establish the existence and geometric ergodicity of periodic
measures for weakly dissipative periodic stochastic differential equations,
including the double well problem mentioned above as an example. The
periodic measures give a rigorous description of the equilibria observed
by physicists and the geometric ergodicity gives the convergence to
and the uniqueness of the periodic measures. The uniqueness is significant
in explaining the transition between the two wells as otherwise there
should be two periodic measures instead of one. However, the current
result does not give the estimate of transition time of $10^{5}$
years. We will study this problem in a different publication where
we have derived and analysed the Kramers' equation (\cite{Kramers})
satisfied by the expected exit time\textbf{. }Here the Kramers' equation
is a parabolic PDE with periodic coefficients while it is an elliptic
PDE in the classical case (\cite{FZZ19}). For autonomous systems
with small noise, study by the large deviation theory (\cite{FW98})
suggests transition would occur at an exponential long time (\cite{Fre00,HIPP13}).
In our problem, the noise is not necessarily small. \textbf{ }

In \cite{ZhaoZheng_RPP_defn}, the authors gave a rigorous definition
of random periodic solutions, objects which can be interpreted as
the periodic counterpart of stationary solutions. Just as there is
an ``equivalence'' (possibly on an enlarged probability space) between
invariant measures and stationary processes (\cite{ArnoldRDS,Ochs_Oseledets_equiv_on_enlarged_prob_space}),
the analogous equivalence between random periodic solution and periodic
measures has been proved in \cite{CFHZ2016}. Specifically, by sampling
the random periodic solution, one can construct a periodic measure.
The existence of a random periodic path was shown for semilinear SDEs
in \cite{CF_HZ_BZ_existence_of_RPP} and \cite{FengWuZhaoRPS}. Numerical
approximations of random periodic paths of SDEs were studied in \cite{FengLiuZhaoNumericalRPP}.
In the case of Markovian random dynamical systems, the equivalence
of ergodicity of periodic measure with the pure imaginary simple eigenvalues
of the infinitesimal generator of the semigroup was established in
\cite{CFHZ2016}.

In this paper, we establish the existence and uniqueness of a limiting
periodic measure for time-periodic Markov processes on locally compact
metric spaces. Moreover, we are interested in the geometric convergence.
The underlying approach leads to the results for SDEs with weakly
dissipative drifts by the means of a Lyapunov function and utilising
the coupling method \cite{Lindvall_Coupling,Thorisson_Coupling,MeynTweedie_MarkovChainStability,MeynTweedie1_Chains,MeynTweedie_MarkovChainStability,MeynTweedie3_Processes}
of Markov chains. Then, inspired by techniques from \cite{MattinglyStuartHigham_ErgodicityForSDEs,DaPratoZabczyk_ErgodicityInfDimensions,DaPratoZabcyzk92,HopfnerLocherbachThieullenTimeDepLocalHormanders},
we give generally verifiable results in which time-periodic weakly
dissipative SDEs, gradient SDEs and Langevin equations possess a unique
(geometric) periodic measure. Coefficients of these equations are
generally non-Lipschitz.

Some of the technical ideas in this paper are motivated also by the
Lyapunov function and discrete Markov chain method in \cite{LocherachHopfnerTPeriodic}.
They studied periodic stochastic differential equations with Lipschitz
coefficients and invariant measures of the grid process on multiple
integrals of the period. Since periodic measures were defined some
years later after \cite{LocherachHopfnerTPeriodic}, the authors were
not aware of periodic measures and their ergodicity in \cite{CFHZ2016},
hence were not able to obtain the uniqueness of periodic measures.
An invariant measure can be obtained by lifting the periodic measure
on a cylinder and considering its average over one period. We would
like to remark that in this paper, we require the diffusion coefficient
to be non-degenerate. We believe this is a technical requirement which
can be relaxed by a locally non-degenerate condition or Hörmander's
condition. But the non-degenerate case studied in this paper is already
applicable in many physical problems such as the stochastic periodic
double well potential problem. We note the objective of this current
work is to introduce the main ingredients and techniques to attain
the existence and uniqueness of periodic measures rather than being
the most general results. We leave the refinement of this paper in
this direction to a later publication. Note that the diffusion of
Langevin equation we investigate is degenerate, but satisfies Hörmander's
condition together with the drift of vector fields.

We expect our approach apply to SPDEs. Our expectation derives from
the existing literature where invariant measures for SPDEs was attained
via a coupling method that is similar to ours in spirit. For instance,
in \cite{Mattingly_Coupling_SPDEs,E_Mattingly_Sinai_Coupling_SPDEs,Kuksin_Shiriykan_coupling_SPDEs,Kuksin_Shiriykan_coupling_SPDEs2},
the respective authors utilised the coupling method to attain an invariant
measure for the 2D SNS (two-dimensional stochastic Navier-Stokes equation).
In fact, it is shown that the convergence of invariant measure for
the 2D SNS equation is geometric in \cite{Hairer_Mattingly_Spectral_Wasserstein}.
Other examples includes \cite{Hairer_CouplingSPDEs} for a class of
degenerate parabolic SPDEs including the complex Ginzburg-Landau equation
and \cite{Mattingly_Coupling_SPDEs} for dissipative SPDEs. We refer
readers to \cite{Mattingly_Progress_SNS} where key aspects to attain
invariant measures via the coupling method in the infinite dimensional
setting was discussed.

In the final section, we prove the Fokker-Planck equation for the
density of a periodic measure. We give explicitly a formula for this
density for periodically forced Ornstein-Uhlenbeck processes.

\section{Preliminaries }

We recall some basic definitions, notation and standard results of
Markovian processes on locally compact separable metric space $(E,\mathcal{B})$
where $\mathcal{B}$ is the natural Borel $\sigma$-algebra and time
indices $\mathbb{T}=\mathbb{N}:=\{0,1,..,\}$ or $\mathbb{R}^{+}$.
By convention, when $\mathbb{T}=\mathbb{N}$, the Markov process is
referred as a Markov chain. The objective of this section is to state
important results from time-homogeneous Markov chain that would be
crucial in proving vital results for $T$-periodic Markovian systems. 

Let $P:\mathbb{T}\times\mathbb{T}\times E\times\mathcal{B}\rightarrow[0,1]$
be a two-parameter Markov transition kernel. It satisfies
\begin{enumerate}
\item $P(s,t,x,\cdot)$ is a probability measure on $\left(E,\mathcal{B}\right)$
for all $s\leq t$ and all $x\in E$.
\item $P(s,t,\cdot,B)$ is a $\mathcal{B}$-measurable function for all
$s\leq t$ and $\Gamma\in\mathcal{B}$.
\item (Chapman-Kolmogorov) For all $s\leq r\leq t$, one has
\[
P(s,t,x,\Gamma)=\int_{E}P(s,r,x,dy)P(r,t,y,\Gamma),\quad x\in E,\Gamma\in\mathcal{B}.
\]
\item $P(s,s,x,B)=1_{\Gamma}(x)$ for all $s\in\mathbb{T}$, $x\in E$ and
$\Gamma\in\mathcal{B}$.
\end{enumerate}
For $s\leq t$, define linear operators $P(s,t)$ acting on $\mathcal{B}_{b}(E)$,
the space of bounded measurable functions by
\begin{align*}
P(s,t)f(x) & =\int_{E}f(y)P(s,t,x,dy),\quad f\in\mathcal{B}_{b}(E),x\in E.
\end{align*}
We say that $P(\cdot,\cdot)$ is Feller if for all $s\leq t$, $P(s,t)f\in C_{b}(E)$
when $f\in C_{b}(E)$ and strong Feller if $P(s,t)f\in C_{b}(E)$
when $f\in\mathcal{B}_{b}(E)$. For $s\leq t$, we define adjoint
operator $P^{*}(s,t)$ acting on $\mathcal{\mathcal{P}}(E)$, the
space of probability measures on $(E,\mathcal{B})$ by

\[
\left(P^{*}(s,t)\mu\right)(\Gamma)=\int_{E}P(s,t,x,\Gamma)\mu(dx),\quad\mu\in\mathcal{P}(E),\Gamma\in\mathcal{B}.
\]

It is well-known that $P(s,t)$ and $P^{*}(s,t)$ forms a two-parameter
semigroup on $\mathcal{B}_{b}(E)$ and $\mathcal{P}(E)$ respectively
and satisfies $P(s,t)=P(s,r)P(r,t)$ and $P^{*}(s,t)=P^{*}(r,t)P^{*}(s,r)$.
On $\mathcal{P}(E)$, we endow the total variation norm defined by
\[
\left\lVert \mu_{1}-\mu_{2}\right\rVert _{TV}:=\sup_{\Gamma\in\mathcal{B}}|\mu_{1}(\Gamma)-\mu_{2}(\Gamma)|,\quad\mu_{1},\mu_{2}\in\mathcal{P}(E).
\]
It is easy to show that $P^{*}(s,t):(\mathcal{P}(E),\lVert\cdot\rVert_{TV})\rightarrow(\mathcal{P}(E),\lVert\cdot\rVert_{TV})$
has operator norm $\lVert P^{*}(s,t)\rVert=1$. While many of the
convergence results presented here holds in other norms than the total
variation norm (such as $f$-norms). For clarity and simplicity, we
shall only consider convergence in the total variation norm. Some
results only require weak convergence of measures. Hence we occasionally
consider $\mu\in\mathcal{\mathcal{P}}(E)$ as a linear functional
on $C_{b}(E)$ by 
\[
\mu(f)=\int_{E}f(x)\mu(dx),\quad f\in C_{b}(E).
\]
And we say $\mu,\nu\in\mathcal{P}(E)$ are equal if $\mu(f)=\nu(f)$
for all $f\in C_{b}(E)$. It is easy to show that $P^{*}(s,t)\mu(f)=\mu(P(s,t)f)$
for $\mu\in\mathcal{P}(E),f\in C_{b}(E)$ and $s\leq t$. 

We give the definition of a time-periodic Markov transition kernel.
We also introduce the stronger definition of minimal time-periodic.
Note that time-periodic Markov kernels depends on initial and terminal
time. 
\begin{defn}
The two-parameter Markov transition kernel $P(\cdot,\cdot,\cdot,\cdot)$
is said to be $T$-periodic for some $T>0$ if
\begin{equation}
P(s,t,x,\cdot)=P(s+T,t+T,x,\cdot),\quad\text{for all }x\in E,s\leq t.\label{eq:T-periodic}
\end{equation}
Moreover, we say $P(\cdot,\cdot,\cdot,\cdot)$ is minimal$T$-periodic
if for every $\delta\in(0,T)\cap\mathbb{T}$
\begin{equation}
P(s,t,x,\cdot)\neq P(s+\delta,t+\delta,x,\cdot),\quad\text{for all }x\in E,s\leq t.\label{eq:MinimalPeriod}
\end{equation}
And we say $P(\cdot,\cdot,\cdot,\cdot)$ is time-homogeneous if 
\[
P(s,t,x,\cdot)=P(0,t-s,x,\cdot),\quad\text{for all }x\in E,s\leq t.
\]
\end{defn}
The definition of $T$-periodic should be clear and intuitive. Observe
that minimal $T$-periodic assumption is stronger. It rules out the
possibility of being time-homogeneous and enforces non-trivial period
for every state. Equation (\ref{eq:T-periodic}) on the other hand
allows states to have trivial period. This implies results of this
paper assuming $T$-periodic $P$ recovers results for the usual time-homogeneous
case. 

As a convention, we denote by $P(t)$ for the time-homogeneous Markov
semigroup and $P^{*}(t)$ for its adjoint depending only on the elapsed
time $0\leq t\in\mathbb{T}$. Specifically for $\mathbb{T}=\mathbb{N}$,
we denote $P:=P(1)$ and $P^{*}:=P^{*}(1)$ for the ``one-step''
semigroup and adjoint semigroup respectively. We now define our central
objects of study characterising stationary and periodic behaviour. 
\begin{defn}
A probability measure $\pi\in\mathcal{P}(E)$ is called an invariant
(probability) measure with respect to $P(s,t)$ if 
\[
P^{*}(s,t)\pi=\pi\quad\text{for all }s\leq t.
\]
When $P$ is time-homogeneous, $\pi$ satisfies $P^{*}(t)\pi=\pi$
for all $t\geq0$. In particular, when $\mathbb{T}=\mathbb{N}$, $\pi$
needs only to satisfy the one-step relation $P^{*}\pi=\pi$. 
\end{defn}
Invariant measures has been well-studied for many decades in many
general settings. For example time-homogeneous Markov chains on finite
dimensional state space \cite{MeynTweedie1_Chains,NorrisMC,MeynTweedie_MarkovChainStability},
and Markov processes on finite state space \cite{StroockMarkovProcesses,NorrisMC},
on infinite dimensional state spaces \cite{DaPratoZabczyk_ErgodicityInfDimensions}.
On the other hand, the formulation of periodic measure below is new
and was first formally defined \cite{CFHZ2016}. 
\begin{defn}
A measure-valued function $\rho:\mathbb{T}\rightarrow\mathcal{P}(E)$
is called a $T$-periodic (probability) measure with respect to $P(\cdot,\cdot)$
if for all $0\leq s\leq t$
\[
\rho_{s+T}=\rho_{s},\quad\rho_{t}=P^{*}(s,t)\rho_{s}.
\]
\end{defn}
Note that periodic measures are invariant measures when the period
is trivial. We shall give sufficient conditions to ensure the periodic
measure has a minimal positive period. In classic literature, see
\cite{DaPratoZabczyk_ErgodicityInfDimensions,Hasminskii,MeynTweedie1_Chains,MeynTweedie_MarkovChainStability,StroockMarkovProcesses,NorrisMC}
for instance, appropriate assumptions yields asymptotic convergence
of the Markov kernel towards a unique invariant measure. However,
these classical asymptotic results seems to have neglected the possibility
of asymptotically periodic behaviour. While conceptually simple, it
seems that asymptotic periodic behaviour was first formally pointed
by Feng and Zhao in \cite{CFHZ2016} and formalised under the definition
of periodic measures. Nonetheless, these limiting invariant measures
results can still be utilised for time-periodic Markovian system.
We end this section by quoting without proof two now-classical results
for time-homogeneous Markov chain result taken as special cases from
\cite{MeynTweedie1_Chains,MeynTweedie_MarkovChainStability}. To state
the results, we require the following definitions.
\begin{defn}
Let $P$ be a one-step time-homogeneous Markov transition kernel.
We say that $P$ satisfies the ``minorisation'' or ``local Doeblin''
condition if there exists a non-empty measurable set $K\in\mathcal{B}$,
constant $\eta\in(0,1]$ and a probability measure $\varphi$ such
that 
\begin{equation}
P(x,\cdot)\geq\eta\varphi(\cdot),\quad x\in K.\label{eq:LocalDoeblin}
\end{equation}
\end{defn}
\begin{defn}
A function $V:\mathbb{T}\times E\rightarrow\mathbb{R}^{+}$ is norm-like
(or coercive) if $V(s,x)\rightarrow\infty$ as $\lVert x\rVert\rightarrow\infty$
for every fixed $s\in\mathbb{T}$ i.e. the level-sets $\{x\in E|V(s,x)\leq r\}$
are pre-compact for each $r>0$. 
\end{defn}
\begin{lem}
\label{lem:-ExistenceAndUniqueness_MC} (Theorem 4.6 \cite{MeynTweedie1_Chains})
Let $P$ be a one-step time-homogeneous Markov transition kernel and
assume there exists a norm-like function $U:E\rightarrow\mathbb{R}^{+}$,
a compact set $K\in\mathcal{B}$ and $\epsilon>0$ such that
\begin{align}
PU-U & \leq-\epsilon\quad\text{on }K^{c},\label{eq:FosterLya1}\\
PU & <\infty\quad\text{on }K.\label{eq:FosterLya2}
\end{align}
Then there exists a unique invariant measure $\pi$ with respect to
$P$. Moreover if $P$ satisfies the local Doeblin condition (\ref{eq:LocalDoeblin})
then the invariant measure is limiting i.e. for any $x\in E$, 
\[
\lVert P^{n}(x,\cdot)-\pi\rVert_{TV}\rightarrow0,\quad\text{as }n\rightarrow\infty.
\]
\end{lem}
In literature, conditions (\ref{eq:FosterLya1}) and (\ref{eq:FosterLya2})
are typically referred as the Foster-Lyapunov drift criteria and has
the interpretation that the process moves inwards on average when
outside the compact set. And$U$ is referred as (Foster-)Lyapunov
function. Lemma \ref{lem:-ExistenceAndUniqueness_MC} is a qualitative
result and does not give any rate of convergence. The following result
from \cite{MeynTweedie1_Chains,MeynTweedie_MarkovChainStability}
gives sufficient condition for a time-homogeneous Markov chain to
possess a unique invariant measure that converges geometrically. 
\begin{lem}
\label{lem:GeometricConvergence-MarkovChains}(Theorem 6.3 \cite{MeynTweedie1_Chains},
Theorem 15.0.1 \cite{MeynTweedie_MarkovChainStability}) Let $P$
be a one-step time-homogeneous Markovian transition kernel satisfying
(\ref{eq:LocalDoeblin}). Assume there exists a norm-like function
$U:E\rightarrow\mathbb{R}^{+}$, $\alpha\in(0,1)$ and $\beta>0$
such that
\begin{equation}
PU\leq\alpha U+\beta\quad\text{on }E.\label{eq:MarkovChainGeometricDriftCondition}
\end{equation}
Then there exists a unique geometric invariant measure $\pi\in\mathcal{P}(E)$
i.e. there exist constants $0<R<\infty$ and $r\in(0,1)$ such that
\[
\left\lVert P^{n}(x,\cdot)-\pi\right\rVert _{TV}\leq R(U(x)+1)r^{n},\quad x\in E,n\in\mathbb{N}.
\]
\end{lem}

\section{Periodic Measures for Time-Periodic Markovian Systems }

The aim of this section is to give new results to establish the existence,
uniqueness and convergence of a periodic measure in the general setting
of time-periodic Markovian systems on a locally compact metric space
$E$. We will give sufficient conditions in which a periodic measure
to have a minimal positive period (hence not an invariant measure)
for $T$-periodic Markov processes on Euclidean space. We start with
the following basic existence and uniqueness lemma. 
\begin{lem}
\label{lem:BasicExistenceUniquePM}Let $P(\cdot,\cdot,\cdot,\cdot)$
be a two-parameter $T$-periodic Markov transition kernel. Assume
for some fixed $s_{*}\in\mathbb{T}$ there exists invariant measure
$\rho_{s_{*}}$ with respect to the one-step Markov transition kernel
$P(s_{*},s_{*}+T)$. Then there exists a $T$-periodic measure $\rho$
with respect to $P(\cdot,\cdot)$. If $\rho_{s_{*}}$ is unique then
$\rho$ is also unique. 
\end{lem}
\begin{proof}
Given $\rho_{s_{*}}$, define the following measures
\begin{equation}
\rho_{s}:=P^{*}(s_{*},s)\rho_{s_{*}},\quad s\geq s_{*}.\label{eq:PushforwardPM}
\end{equation}
Extend $\rho_{s}$ by periodicity for $s\leq s_{*}$. Then it is clear
that $\rho:\mathbb{T}\rightarrow\mathcal{P}(E)$ and is easy to show
that $\rho$ is a periodic measure with respect to $P(\cdot,\cdot).$
Now suppose that $\rho_{s_{*}}$ is the unique invariant measure with
respect to $P(s_{*},s_{*}+T)$, we prove $\rho$ is also unique. Suppose
there are two $T$-periodic measures $\rho^{i}=\left(\rho_{s}^{i}\right)_{s\in\mathbb{T}}$
for $i=1,2$ with respect to $P(\cdot,\cdot)$. By definition of periodic
measures, $\rho_{s}^{i}$ satisfies $\rho_{s}^{i}=P^{*}(s_{*},s)\rho_{s_{*}}^{i}$
for $s\geq s_{*}$, hence by the linearity of $P^{*}$ 
\[
\lVert\rho_{s}^{1}-\rho_{s}^{2}\rVert_{TV}=\lVert P^{*}(s_{*},s)(\rho_{s_{*}}^{1}-\rho_{s_{*}}^{2})\rVert_{TV}\leq\lVert P^{*}(s_{*},s)\rVert\lVert\rho_{s_{*}}^{1}-\rho_{s_{*}}^{2}\rVert_{TV}
\]
The result follows by the assumption of uniqueness i.e. $\rho_{s_{*}}^{1}=\rho_{s_{*}}^{2}$
.
\end{proof}
Note that, by definition, an invariant measure is always a periodic
measure with a trivial period. For applications, we expect it is important
to distinguish periodic measures of minimal positive period and those
of a trivial period. However, it is not immediate whether the periodic
measure constructed in Lemma \ref{lem:BasicExistenceUniquePM} has
a trivial period. The distinction can be subtle as periodic coefficients
does not immediately yield a periodic measure with a non-trivial period.
For example, let $W_{t}$ be a one-dimensional Brownian motion and
$S$ is a continuously differentiable $T$-periodic function and consider
the following SDE 
\[
dX_{t}=(-\alpha X_{t}^{3}+S(t)X_{t})dt+\sigma dW_{t},\quad\alpha>0,\sigma\neq0.
\]
The results from Section \ref{sec:SDEs} of this paper yields the
existence and uniqueness of a periodic measure with a minial positive
period. On the other hand, the same SDE with multiplicative linear
noise, 
\[
dX_{t}=(-\alpha X_{t}^{3}+S(t)X_{t})dt+X_{t}dW_{t},\quad\alpha>0,
\]
has $\delta_{0}$ (Dirac mass at the origin) as both the invariant
measure periodic measure with a trivial period of the system. In fact,
for time-homogeneous Markovian systems in general, the time-average
of a periodic measure always yields an invariant measure. This is
known for Markovian cocycles in the general framework of random dynamical
systems (RDS); formal statements and proof can be found in \cite{CFHZ2016}.
This means that $T$-periodic measures and invariant measures coexist
for time-homogeneous Markovian systems. However, for time-inhomogeneous
and specifically time-periodic Markovian systems, invariant measures
and periodic measures can be mutually exclusive. Should the measures
be mutually exclusive, this has the important implication that the
long term behaviour is characterised by strictly periodic behaviour.
The following proposition gives sufficient condition in which an invariant
measure cannot exist, thus if there exists any periodic measure, it
has a minimal positive period. We only state and prove it for Euclidean
state space, it will be apparent that it can hold in more general
spaces. 
\begin{prop}
\label{prop:CoExistence_with_IM} Let $T>0$ and $E=\mathbb{R}^{d}$
and assume $P$ is minimal $T$-periodic and strong Feller. Then if
there exist periodic measure(s), it has a minimal positive period.
\end{prop}
\begin{proof}
We prove by contradiction and assume there exists periodic measure
with a trivial period i.e. there exists an invariant measure $\pi\in\mathcal{P}(\mathbb{R}^{d})$.
Then it must be that for all fixed $\delta\geq0$
\[
P^{*}(s,t)\pi=\pi=P^{*}(s+\delta,t+\delta)\pi.
\]
By duality and linearity, this is equivalent to
\begin{equation}
\pi(P(s,t)f-P(s+\delta,t+\delta)f)=0,\quad\text{for all }f\in C_{b}(\mathbb{R}^{d}).\label{eq:NoInvariantMeasure}
\end{equation}
To prove the result, we construct an $f\in C_{b}(\mathbb{R}^{d})$
such that 
\[
\pi\left(P(s,t)f-P(s+\delta,t+\delta)f\right)>0.
\]
By minimal $T$-periodic assumption (\ref{eq:MinimalPeriod}), for
every fixed $x\in\mathbb{R}^{d}$, there exists $f_{x}\in C_{b}(\mathbb{R}^{d})$
such that $P(s,t,x,\cdot)(f_{x})\neq P(s+\delta,t+\delta,x,\cdot)(f_{x})$.
Hence, without loss of generality, there exists $\epsilon>0$ such
that
\[
P(s,t)f_{x}(x)-P(s+\delta,t+\delta)f_{x}(x)>2\epsilon.
\]
For $x=(x_{1},..,x_{d})\in\mathbb{R}^{d}$, define the half-open cube
of length $r>0$ centred at $x$ by
\[
C(x,r):=\prod_{i=1}^{d}\left[x_{i}-r,x_{i}+r\right)\in\mathcal{B}(\mathbb{R}^{d}),
\]
where $\prod$ denotes the standard Cartestian product of Euclidean
space. The (not necessarily strong) Feller assumption yields that
$P(s,t)f_{x}\in C_{b}(\mathbb{R}^{d})$ and $P(s+\delta,t+\delta)f_{x}\in C_{b}(\mathbb{R}^{d})$
hence there exists a $r_{x}>0$ such that
\[
P(s,t)f_{x}(y)-P(s+\delta,t+\delta)f_{x}(y)>\epsilon,\quad\text{for all }y\in C(x,r_{x}).
\]
Let $r:=\min_{x}r_{x}>0$ and for $\alpha=(\alpha_{1},...,\alpha_{d})\in\mathbb{Z}^{d}$
define the half-open cube
\[
C_{\alpha}=\prod_{i=1}^{d}\left[\alpha_{i}r,(\alpha_{i}+1)r\right)\in\mathcal{B}(\mathbb{R}^{d}).
\]
Clearly $\{C_{\alpha}\}_{\alpha\in\mathbb{Z}^{d}}$ is a countable
disjoint covering of $\mathbb{R}^{d}$. For every $\alpha\in\mathbb{Z}^{d}$,
define $x_{\alpha}\in\mathbb{R}^{d}$ with components $x_{i}=\left(\alpha_{i}+0.5\right)r$
for each $1\leq i\leq d$ i.e. the midpoint of the cube and define
the functions 
\[
f_{\alpha}:=f_{x_{\alpha}},\quad\alpha\in\mathbb{Z}^{d}.
\]
And define the piecewise continuous functions over all possible tuples
\[
f(x):=\begin{cases}
f_{(1,0,...,0)}(x) & \text{if }x\in C_{(1,0,...,0)},\\
f_{(0,1,...,0)}(x) & \text{if }x\in C_{(0,1,...,0)},\\
\vdots & \vdots\\
f_{\alpha}(x) & \text{if }x\in C_{\alpha}.
\end{cases}
\]
By construction, $f\in\mathcal{B}_{b}(\mathbb{R}^{d})$. Let $g:=P(s,t)f-P(s+\delta,t+\delta)f$.
Then $g>\epsilon$ and the strong Feller assumption implies that $g\in C_{b}(\mathbb{R}^{d})$.
Hence for any $\pi\in\mathcal{P}(\mathbb{R}^{d})$, $\pi(g)>0$ contradicting
(\ref{eq:NoInvariantMeasure}).
\end{proof}
We make the following trivial but important observation. If $(X_{t})_{t\in\mathbb{T}}$
is a $T$-periodic Markov process, then $(Z_{n}^{s})_{n\in\mathbb{N}}:=(X_{s+nT})_{n\in\mathbb{N}}$
is a time-homogeneous Markov chain. This enables the usage of classical
time-homogeneous Markov chain theory that is already well-established.
Beyond the theoretical advantage, this observation is practically
important in applications. 

We can now discuss ergodicity of time-periodic Markovian systems.
Classically, ergodic (time-homogeneous) Markov processes has the property
that the Markov transition kernel converges to an invariant measure
as time tends to infinity. In this sense, the invariant measure characterises
the long-time behaviour of the system. On the other hand, for periodic
measure (with a minimal positive period) cannot be limiting in the
same way because the periodic measures evolves over time. However,
it is possible that the Markov transition kernel can converge along
integral multiples for $T$-periodic Markovian processes. This captures
the idea that the periodic measure describes long-time periodic behaviour
of the system. This shall be apparent and rigorously written in the
forthcoming theorem. We remark also that the forthcoming theorem can
be regarded as the time-periodic generalisation of Lemma \ref{lem:-ExistenceAndUniqueness_MC}.
\begin{thm}
\label{thm:Limiting_PM_on_E} Let $P$ be a $T$-periodic Markov transition
kernel and assume there exists $s_{*}\in\mathbb{T}$, norm-like function
$U_{s_{*}}:E\rightarrow\mathbb{R}^{+}$, a non-empty compact set $K\in\mathcal{B}$,
$\epsilon>0$, $\eta_{s_{*}}\in(0,1]$, $\varphi_{s_{*}}\in\mathcal{P}(E)$
such that
\begin{align}
P(s_{*},s_{*}+T)U_{s_{*}}-U_{s_{*}} & \leq-\epsilon\quad\text{on }K^{c},\label{eq:FosterLya1_s*}\\
P(s_{*},s_{*}+T)U_{s_{*}} & <\infty\quad\text{on }K,\label{eq:FosterLya2_s*}\\
P(s_{*},s_{*}+T,x,\cdot) & \geq\eta_{s_{*}}\varphi_{s_{*}}(\cdot),\quad x\in K.\label{eq:LDC_s*}
\end{align}
i.e. (\ref{eq:LocalDoeblin}), (\ref{eq:FosterLya1}) and (\ref{eq:FosterLya2})
are satisfied for $P(s_{*},s_{*}+T)$. Then there exists a unique
$T$-periodic measure $\rho$ that satisfies all the convergences
below.
\begin{enumerate}
\item For any fixed $x\in E$ and $s\in\mathbb{T}$
\begin{equation}
\lVert P(s,s+nT,x,\cdot)-\rho_{s}\rVert_{TV}\rightarrow0,\quad\text{as }n\rightarrow\infty.\label{eq:LimitingPMAlongPeriod}
\end{equation}
\item For any fixed $x\in E$ and $s\in\mathbb{T}$, the following ``moving''
convergence holds,
\begin{equation}
\left\lVert P(s,t,x,\cdot)-\rho_{t}\right\rVert _{TV}=0,\quad\text{as }t\rightarrow\infty.\label{eq:MovingConvergence}
\end{equation}
\item Allowing for negative initial time, for any fixed $x\in E$, $s,t\in\mathbb{T}$,
the following pullback convergence holds
\begin{equation}
\left\lVert P(s-nT,t,x,\cdot)-\rho_{t}\right\rVert _{TV}=0\quad\text{as }n\rightarrow\infty.\label{eq:PullbackConvergence_to_rho_t}
\end{equation}
\end{enumerate}
\end{thm}
\begin{proof}
Since $P(s,s+T,\cdot,\cdot)$ is a one-step time-homogeneous Markov
kernel for all $s\in\mathbb{T}$, by Lemma \ref{lem:-ExistenceAndUniqueness_MC},
there exists a unique $\rho_{s_{*}}\in\mathcal{P}(E)$ with respect
to $P(s_{*},s_{*}+T$). Moreover by Lemma \ref{lem:BasicExistenceUniquePM},
there exists a unique periodic measure $\rho$. To show the convergences,
we show that $P(s,s+T)$ satisfies (\ref{eq:LocalDoeblin}), (\ref{eq:FosterLya1})
and (\ref{eq:FosterLya2}) for all $s\in\mathbb{T}$. By $T$-periodicity
of $P$ and the semigroup properties of $P$, observe that
\[
P(s,s_{*})P(s_{*},s_{*}+T)=P(s,s_{*})P(s_{*},s+T)P(s+T,s_{*}+T)=P(s,s+T)P(s,s_{*}),\quad s\leq s_{*}.
\]
Hence applying $P(s,s_{*})$ to both sides of (\ref{eq:FosterLya1_s*})
yields 
\[
P(s,s+T)P(s,s_{*})U_{s_{*}}-P(s,s_{*})U_{s_{*}}<-\epsilon,\quad\text{on }K^{c}.
\]
i.e. $U_{s}:=P(s,s_{*})U_{s_{*}}$ satisfies (\ref{eq:FosterLya1})
with respect to $P(s,s+T)$. Analogously, $U_{s}$ satisfies (\ref{eq:FosterLya2}).
It is easy to verify that $U_{s}\geq0$. We extend $U_{s}$ for all
$s\in\mathbb{T}$ by periodicity. We claim that for any $s\geq s_{*}$,$\eta_{s}:=\eta_{s_{*}}\in(0,1]$
and $\varphi_{s}:=P^{*}(s_{*},s)\varphi_{s_{*}}\in\mathcal{P}(E)$
satisfies
\begin{equation}
P(s,s+T,x,\cdot)\geq\eta_{s}\varphi_{s}(\cdot),\quad x\in K.\label{eq:LocalDoeblinOtherTimeValues}
\end{equation}
i.e. $P(s,s+T)$ satisfies (\ref{eq:LocalDoeblin}). Should this not
be the case, then there exists some $x\in K$ and $\Gamma\in\mathcal{B}$
such that $P(s,s+T,x,\Gamma)<\eta_{s}\varphi_{s}(\Gamma)$. Then 

\[
P(s_{*},s)P(s,s+T,x,\Gamma)=P(s_{*},s+T,x,\Gamma)<\eta_{s}\varphi_{s}(\Gamma),
\]
by applying $P(s_{*},s)$ to both sides and Chapman-Kolmogorov equation.
However by assumption (\ref{eq:LDC_s*}),
\begin{align*}
\eta_{s}\varphi_{s}(\Gamma) & >P(s_{*},s+T,x,\Gamma)\\
 & =P^{*}(s_{*}+T,s+T)P(s_{*},s_{*}+T,x,\Gamma)\\
 & =\eta_{s_{*}}P^{*}(s_{*},s)\varphi_{s_{*}}(\Gamma),
\end{align*}
which is a contradiction. We again extend by periodicity for all $s\in\mathbb{T}$.
Thus, the assumptions of Lemma \ref{lem:-ExistenceAndUniqueness_MC}
are satisfied to deduce (\ref{eq:LimitingPMAlongPeriod}) for all
$s\in\mathbb{T}$. Observe that for $t\geq s+nT$, 
\begin{align*}
\lVert P(s,t,x,\cdot)-\rho_{t}\rVert_{TV} & =\lVert P^{*}(s+nT,t)P(s,s+nT,x,\cdot)-P^{*}(s+nT,t)\rho_{s+nT}\rVert_{TV}\\
 & =\lVert P^{*}(s+nT,t)P(s,s+nT,x,\cdot)-P^{*}(s+nT,t)\rho_{s}\rVert_{TV}\\
 & \leq\lVert P(s,s+nT,x,\cdot)-\rho_{s}\rVert_{TV}.
\end{align*}
Hence (\ref{eq:MovingConvergence}) follows by (\ref{eq:LimitingPMAlongPeriod}),
by taking $t\rightarrow\infty$ followed by $n\rightarrow\infty$.
Using (\ref{eq:LimitingPMAlongPeriod})., convergence (\ref{eq:PullbackConvergence_to_rho_t})
holds due to 
\[
P(s-nT,t,x,\cdot)=P(s,t+nT,x,\cdot)=P^{*}(s+nT,t+nT)P(s,s+nT,x,\cdot)=P^{*}(s,t)P(s,s+nT,x,\cdot).
\]
\end{proof}
We elaborate on the convergences given in Theorem \ref{thm:Limiting_PM_on_E}.
The first convergence (\ref{eq:LimitingPMAlongPeriod}) is clear where
the convergence is along integral multiples of the period towards
a fixed measure. That is, ergodicity of the grid chain. Convergence
(\ref{eq:MovingConvergence}) extends (\ref{eq:LimitingPMAlongPeriod})
by allowing the convergence to be taken continuously in time. Observe
that equation (\ref{eq:MovingConvergence}) captures the idea that
long-term behaviour is characterised by the periodic measure. Note
that this convergence is towards a ``moving target'' as the periodic
measure evolves over time. It is typical in the theory of non-autonomous
dynamical systems \cite{NonautonDS} and RDS (random dynamical systems)
\cite{NonautonRDS} to study ``pullback'' convergence. This is convergence
where one takes initial time further and further back in time rather
than the forward time. The advantage is that the convergence will
be to a fixed target rather than a moving one. This is the content
of convergence (\ref{eq:PullbackConvergence_to_rho_t}). In general,
(forward) convergence and pullback convergence do not coincide (see
\cite{NonautonDS,NonautonRDS} for examples). In this $T$-periodic
case, we see that the convergences coincide. 

Assuming we have a stochastic Lyapunov function for a $T$-periodic
Markovian kernel, Theorem \ref{thm:Limiting_PM_on_E} gives a limiting
periodic measure provided the local Doeblin condition (\ref{eq:LDC_s*}).
The following two results gives sufficient conditions in which (\ref{eq:LDC_s*})
holds. We denote for convenience $\mathcal{M}(E)$ to be the space
of measures on $(E,\mathcal{B})$. 
\begin{prop}
\label{prop:PointwiseNondegDensity}Let $P$ be a $T$-periodic Markov
transition kernel and assume there exists some $s_{*}\in\mathbb{T}$,
a non-empty set $K\in\mathcal{B}$, $\epsilon>0$ and $\Lambda\in\mathcal{M}(E)$
such that $\Lambda(K)>0$, $P(s,t,x,\cdot)$ possesses a density $p(s,t,x,y)$
with respect to $\Lambda$ and 
\begin{equation}
\inf_{x,y\in K}p(s_{*},s_{*}+T,x,y)>0.\label{eq:PositiveDensityInCompact}
\end{equation}
Then the local Doeblin condition (\ref{eq:LDC_s*}) of Theorem \ref{thm:Limiting_PM_on_E}
holds.
\end{prop}
\begin{proof}
By Theorem \ref{thm:Limiting_PM_on_E}, it suffices to show $P(s_{*},s_{*}+T)$
satisfies (\ref{eq:LocalDoeblin}) for some $s_{*}\in\mathbb{T}$.
By assumption that $\Lambda(K)>0$, 
\begin{align*}
\eta: & =\int_{E}\inf_{x\in K}p(s_{*},s_{*}+T,x,y)dy\\
 & \geq\int_{K}\inf_{x\in K}p(s_{*},s_{*}+T,x,y)dy\\
 & =\inf_{x,y\in K}p(s_{*},s_{*}+T,x,y)\Lambda(K)\\
 & >0.
\end{align*}
Clearly, $\eta\in(0,1]$. Define for any $\Gamma\in\mathcal{B}$,
\[
\varphi(\Gamma):=\frac{1}{\eta}\int_{\Gamma}\inf_{x\in K}p(s_{*},s_{*}+T,x,y)dy.
\]
It is easy to verify that $\varphi\in\mathcal{P}(E)$ and for any
$x\in K$ and any $\Gamma\in\mathcal{B}$
\begin{align*}
P(s_{*},s_{*}+T,x,\Gamma) & =\int_{\Gamma}p(s_{*},s_{*}+T,x,y)dy\\
 & \geq\int_{\Gamma}\inf_{x\in K}p(s_{*},s_{*}+T,x,y)dy\\
 & =\eta\varphi(\Gamma).
\end{align*}
Thereby (\ref{eq:LocalDoeblin}) holds with constant $\eta$ and probability
measure $\varphi$. 
\end{proof}
In practice, assumption (\ref{eq:PositiveDensityInCompact}) in Proposition
\ref{prop:PointwiseNondegDensity} can be difficult to verify as well
as being stronger than required. By assuming the Markov transition
kernel possesses a continuous density, we can relax (\ref{eq:PositiveDensityInCompact}).
For the forthcoming theorem, we define $\mathcal{M}^{+}(E)=\{\mu\in\mathcal{M}(E)|\mu(\Gamma)>0,\quad\text{non-empty open }\Gamma\in\mathcal{B}\}$.
We will make explicit use of the metric $d$ on $(E,\mathcal{B})$
and define $B_{r}(x):=\{y\in E|d(x,y)<r\}$ to be the open ball of
radius $r>0$ centred at $x\in E$. 
\begin{thm}
\label{thm:DensityContinuityforLCD}Let $P$ be a $T$-periodic Markov
transition kernel and assume there exists some $s_{*}\in\mathbb{T}$,
a non-empty compact set $K\in\mathcal{B}$, $0\leq r\le T$ and $\Lambda\in\mathcal{M}^{+}(E)$
such that $P(s,t,x,\cdot)$ possesses a (local) density $p(s,t,x,y)$
with respect to $\Lambda$ and is jointly continuous on $K\times K$.
Assume further that for any non-empty open set $\Gamma_{1},\Gamma_{2}\subset K$
and $x\in K$ 
\begin{equation}
P(s_{*},s_{*}+r,x,\Gamma_{1})>0,\quad P(s_{*}+r,s_{*}+T,x,\Gamma_{2})>0.\label{eq:LocalIrred}
\end{equation}
Then the local Doeblin condition (\ref{eq:LDC_s*}) of Theorem \ref{thm:Limiting_PM_on_E}
holds.
\end{thm}
\begin{proof}
Fix any $y'\in K$, by (\ref{eq:LocalIrred}), then for any non-empty
open set $\Gamma\subset K$, 
\[
P(s_{*}+r,s_{*}+T,y',\Gamma)>0.
\]
By the existence of a density, there exists $z'\in\Gamma$ such that
\[
p(s_{*}+r,s_{*}+T,y',z')\geq2\epsilon,
\]
for some $\epsilon>0$. Joint continuity assumption implies there
exists $r_{1},r_{2}>0$ such that
\[
p(s_{*}+r,s_{*}+T,y,z)\geq\epsilon,\quad\text{for all }y\in B_{r_{1}}(y')\subset K,z\in B_{r_{2}}(z')\subset K.
\]
Hence for any $\Gamma\in\mathcal{B}$ and $y\in B_{r_{1}}(y')$, 
\begin{align*}
P(s_{*}+r,s_{*}+T,y,\Gamma) & =\int_{\Gamma}p(s_{*}+r,s_{*}+T,y,z)dz\\
 & \geq\int_{\Gamma\cap B_{r_{2}}(z')}p(s_{*}+r,s_{*}+T,y,z)dz\\
 & \geq\epsilon\Lambda(\Gamma\cap B_{r_{2}}(z')).
\end{align*}
By (\ref{eq:LocalIrred}), we have 
\[
P(s_{*},s_{*}+r,x,B_{r_{1}}(y'))>0,\quad\text{for all }x\in K.
\]
As $p(s_{*},s_{*}+r,x,y)$ is a continuous function of $x$, by dominated
convergence theorem, $P(s_{*},s_{*}+r,x,\Gamma)$ is also continuous
function of $x$. Hence, the compactness of $K$ yields 
\[
\inf_{x\in K}P(s_{*},s_{*}+r,x,B_{r_{1}}(y'))\geq\gamma',
\]
for some $\gamma'>0$. In particular, 
\[
\inf_{x\in K}P(s_{*},s_{*}+r,x,B_{r_{1}}(y'))\geq\gamma:=\min\left\{ \gamma',\frac{1}{\epsilon\Lambda(B_{r_{2}}(z'))}\right\} .
\]
Putting them together via Chapman-Kolmogorov equation, we have for
any $x\in K$ and $\Gamma\in\mathcal{B}$,
\begin{align*}
P(s_{*},s_{*}+T,x,\Gamma) & =\int_{E}P(s_{*}+r,s_{*}+T,y,\Gamma)p(s_{*},s_{*}+r,x,y)dy\\
 & \geq\int_{B_{r_{1}}(y')}P(s_{*}+r,s_{*}+T,y,\Gamma)p(s_{*},s_{*}+r,x,y)dy\\
 & \geq\epsilon\Lambda(\Gamma\cap B_{r_{2}}(z'))\int_{B_{r_{1}}(y')}p(s_{*},s_{*}+r,x,y)dy\\
 & =\epsilon\Lambda(\Gamma\cap B_{r_{2}}(z'))P(s_{*},s_{*}+r,x,B_{r_{1}}(y'))\\
 & \geq\epsilon\gamma\Lambda(\Gamma\cap B_{r_{2}}(z')).
\end{align*}
Thus, the probability measure 
\[
\varphi(\cdot)=\frac{\Lambda(\cdot\cap B_{r_{2}}(z'))}{\Lambda(B_{r_{2}}(z'))},
\]
and the constant $\eta=\epsilon\gamma\Lambda(B_{r_{2}}(z'))\in(0,1]$
collectively satisfy the local Doeblin condition (\ref{eq:LDC_s*}).
\end{proof}
\begin{rem}
\label{rem:HaarMeasure}Note that in Theorem \ref{thm:DensityContinuityforLCD},
if $E$ is a locally compact metrisable topological group, then any
Haar measure $\Lambda$ (for which a local density exist and jointly
continuous) will suffice.
\end{rem}
Similar to Theorem \ref{thm:Limiting_PM_on_E}, we end this section
with a theorem for the existence and uniqueness of a geometric periodic
measure. Observe in the theorem that the geometric convergence intrinsically
depends on the initial time and state. This is akin to the autonomous
case where the convergence depends on initial state. 
\begin{thm}
\label{thm:Geometric_PM_on_E} Let $P$ be a $T$-periodic Markov
transition kernel and assume there exists $s_{*}\in\mathbb{T}$, a
norm-like function $U_{s_{*}}:E\rightarrow\mathbb{R}^{+}$, a non-empty
compact set $K\in\mathcal{B}$, $\epsilon>0$ such that $P(s_{*},s_{*}+T)$
satisfies the local Doeblin condition (\ref{eq:LDC_s*}) and there
exist constants $\alpha\in(0,1)$ and $\beta>0$ satisfying 
\[
P(s_{*},s_{*}+T)U_{s_{*}}\leq\alpha U_{s_{*}}+\beta,\quad\text{on }E.
\]
Then there exists a unique geometric periodic measure $\rho$. Specifically,
there exists a norm-like function $V:\mathbb{T}\times E\rightarrow\mathbb{R}^{+}$
constants $R_{s}<\infty$ and $r_{s}\in(0,1)$ such that the following
all holds
\begin{enumerate}
\item For any $s\in\mathbb{T}$ and $x\in E$, we have 
\begin{equation}
\left\lVert P(s,s,+nT,x,\cdot)-\rho_{s}\right\rVert _{TV}\leq R(V(s,x)+1)r_{s}^{n},\quad n\in\mathbb{N}.\label{eq:Geometric_Convergence_Integral_Periods}
\end{equation}
\item For any $s\leq t$, $x\in E$, we have 
\[
\left\lVert P(s,t,x,\cdot)-\rho_{t}\right\rVert _{TV}\leq R_{s}(V(s,x)+1)r_{s}^{n},\quad\mathbb{N}\ni n\leq\lfloor\frac{t-s}{T}\rfloor.
\]
\item Allowing for negative initial time, for any $s\leq t$, $x\in E$,
we have 
\[
\left\lVert P(s-nT,t,x,\cdot)-\rho_{t}\right\rVert _{TV}\leq R_{s}(V(s,x)+1)r_{s}^{n},\quad\mathbb{N}\ni n\leq\lfloor\frac{t-s}{T}\rfloor.
\]
\item The periodic measure is uniformly geometric convergence over initial
time i.e. there exist constants $R>0,r\in(0,1)$ and a norm-like function
$V:E\rightarrow\mathbb{R}^{+}$ such that 
\begin{equation}
\left\lVert P(s,s+nT,x,\cdot)-\rho_{s}\right\rVert _{TV}\leq R(V(x)+1)r^{n},\quad\text{for all }x\in\mathbb{R}^{d},s\in\mathbb{T},n\in\mathbb{N}.\label{eq:ConvToPM_uniform}
\end{equation}
\end{enumerate}
\end{thm}
\begin{proof}
Define $V(s,x):=P(s,s_{*})U_{s_{*}}(x)$ for all $s\leq s_{*}$ and
extend by periodicity for all $s\in\mathbb{T}$. Then analogous to
Theorem \ref{thm:Limiting_PM_on_E}, the function $V(s,\cdot)$ satisfies
(\ref{eq:MarkovChainGeometricDriftCondition}) with respect to $P(s,s+T)$.
Likewise from Theorem \ref{thm:Limiting_PM_on_E}, the local Doeblin
condition holds. Then (\ref{eq:Geometric_Convergence_Integral_Periods})
holds immediately by Lemma \ref{lem:GeometricConvergence-MarkovChains}.
Convergence (\ref{eq:ConvToPM_uniform}) is obvious from (\ref{eq:Geometric_Convergence_Integral_Periods})
by defining $R=\sup_{s\in[0,T)}R_{s}$, $V(x)=\sup_{s\in[0,T)}V(s,x)$
and $r=\sup_{s\in[0,T]}r_{s}<1$. The remaining converges are proven
in the way as Theorem \ref{thm:Limiting_PM_on_E}. 
\end{proof}

\section{\label{sec:SDEs}Time-Periodic Stochastic Differential Equations}

\subsection{Limiting Periodic Measures}

Using the developed theory from Section 3, we apply the results specifically
in the context of $T$-periodic SDEs evolving on Euclidean state space.
In this subsection, we are particularly interested in results that
can be verified  to possess a limiting periodic measure. We will study
the $T$-periodic SDEs with white noises as its source of randomness.
We note however that Theorem \ref{thm:Limiting_PM_on_E} can accommodate
other types of noise. For instance, Höpfner and Löcherbach \cite{LocherachHopfnerTPeriodic}
studied at a periodically forced Ornstein-Uhlenbeck process under
the influence of Lévy noise. Generalising to Lévy noise would be an
area of future works and foresee applications to time-periodic financial
models with jumps. 

We fix some nomenclature and notation. Non-autonomous refer to SDE
coefficients which depend explicitly on time. We always denote by
$(\mathbb{R}^{d},\mathcal{B}(\mathbb{R}^{d}))$ the Euclidean space
where $\mathcal{B}(\mathbb{R}^{d})$ denote the standard Borel $\sigma$-algebra
on $\mathbb{R}^{d}$ and let $\langle\cdot,\cdot\rangle$ and $\lVert\cdot\rVert$
to denote the standard inner-product and norm on $\mathbb{R}^{d}$.
Then we can define $B_{r}(y):=\{x\in\mathbb{\mathbb{R}}^{d}|\left\lVert x-y\right\rVert <r\}$
for the open ball of radius $r>0$ centred at $y$. And denote for
convenience $B_{r}:=B_{r}(0)$. On $\mathbb{R}^{d}$, we re-use $\Lambda$
as the Lebesgue measure. We let $GL(\mathbb{R}^{d})$ denote the space
of invertible $d\times d$ matrices and let $L_{2}(\mathbb{R}^{d}):=\{\sigma\in\mathbb{R}^{d\times d}|\lVert\sigma\rVert_{2}<\infty\}$
where $\lVert\sigma\rVert_{2}=\sqrt{\text{Tr}(\sigma\sigma^{T})}=\sqrt{\sum_{i,j=1}^{d}\sigma_{ij}^{2}}$
as the standard Frobenius norm. 

We let $C^{1,2}(\mathbb{R}^{+}\times\mathbb{R}^{d})$ denote the space
of functions which are continuously differentiable in the first variable
and twice differentiable in the spatial variables and $C_{b}^{\infty}(B_{n})$
denote the space of bounded infinitely differentiable real-valued
functions on $B_{n}$. Functions $b:\mathbb{R}^{+}\times\mathbb{R}^{d}\rightarrow\mathbb{R}^{d}$
and $\sigma:\mathbb{R}^{+}\times\mathbb{R}^{d}\rightarrow\mathbb{R}^{d\times d}$
are said to be locally Lipschitz if for any compact set $K\subset\mathcal{B}(\mathbb{R}^{d})$
there exists a constants $L=L(K)$ and $M=M(K)$ such that $\left\lVert b(t,x)-b(t,y)\right\rVert \leq L\left\lVert x-y\right\rVert $
and $\left\lVert \sigma(t,x)-\sigma(t,y)\right\rVert _{2}\leq M\left\lVert x-y\right\rVert _{2}$
for $x,y\in K$. They are (globally) Lipschitz if $K=E$. We say that
$\sigma$ has linear growth if there exists a constant $C>0$ such
that

\begin{equation}
\lVert\sigma(t,x)\rVert_{2}^{2}\leq C(1+\lVert x\rVert^{2}),\quad t\in\mathbb{R}^{+},x\in\mathbb{R}^{d}.\label{eq:LinearGrowth_sigma}
\end{equation}
Define for ease, $\lVert\sigma\rVert_{\infty}:=\sup_{(t,x)\in\mathbb{R}^{+}\times\mathbb{R}^{d}}\lVert\sigma(t,x)\rVert_{2}$
and say $\sigma$ is bounded with bounded inverse if
\begin{equation}
\max\{\lVert\sigma\rVert_{\infty},\lVert\sigma^{-1}\rVert_{\infty}\}<\infty.\label{eq:Sigma_and_inverse_bounded}
\end{equation}
For tuple $\alpha=(\alpha_{0},\alpha_{1},...,\alpha_{d})\in\mathbb{N}^{d+1},$
define the partial derivatives $\partial^{\alpha}=\frac{\partial^{\lvert\alpha\rvert}}{\partial_{t}^{\alpha_{0}}\partial_{x_{1}}^{\alpha_{1}}\cdot\cdot\cdot\partial_{x_{d}}^{\alpha_{d}}}$
where $\lvert\alpha\rvert=\sum_{i=0}^{d}\alpha_{i}$. We say that
the functions $b:\mathbb{R}^{+}\times\mathbb{R}^{d}\rightarrow\mathbb{R}^{d}$
and $\sigma:\mathbb{R}^{d}\rightarrow\mathbb{R}^{d\times d}$ are
locally smooth and bounded if for all $n\in\mathbb{N}$
\begin{equation}
\sigma_{ij}\in C_{b}^{\infty}(B_{n}),\quad1\leq i,j\leq d,\label{eq:LocallySmoothSigma}
\end{equation}
and
\begin{equation}
b(t,x)+\partial^{\alpha}b(t,x)\quad\text{bounded on }\mathbb{R}^{+}\times B_{n},\alpha\in\mathbb{N}^{d+1},\lvert\alpha\rvert=d.\label{eq:LocallySmoothDrift}
\end{equation}
Note that (\ref{eq:LocallySmoothSigma}) and (\ref{eq:LocallySmoothDrift})
imply the respective functions are locally Lipschitz. Whenever we
assume (\ref{eq:LocallySmoothSigma}), we always demand that $\sigma$
is a function of spatial variables only. 

We study Markov processes $X_{t}=X_{t}^{s,x}$ satisfying $T$-periodic
SDEs of the form
\begin{equation}
\begin{cases}
dX_{t}=b(t,X_{t})dt+\sigma(t,X_{t})dW_{t},\\
X_{s}=x.
\end{cases}\label{eq:NonAutonSDE}
\end{equation}
Here $x\in\mathbb{R}^{d}$, $T>0$ and functions $b\in C(\mathbb{R}^{+}\times\mathbb{R}^{d},\mathbb{R}^{d})$
and $\sigma\in C(\mathbb{R}^{+}\times\mathbb{R}^{d},GL(\mathbb{R}^{d}))$
are both $T$-periodic i.e. 
\[
b(t,\cdot)=b(t+T,\cdot),\quad\text{and }\quad\sigma(t,\cdot)=\sigma(t+T,\cdot),
\]
and $W_{t}$ is a $d$-dimensional Brownian motion on the probability
space $(\Omega,\mathcal{F},\mathbb{P})$. The infinitesimal generator
of (\ref{eq:NonAutonSDE}), $\mathcal{L}(t)$ given by 
\begin{equation}
\mathcal{L}(t)f(t,x)=\partial_{t}f(t,x)+\sum_{i=1}^{d}b_{i}(t,x)\partial_{i}f(t,x)+\frac{1}{2}\sum_{i,j=1}^{d}\left(\sigma\sigma^{T}\right)_{ij}\partial_{ij}^{2}f(t,x),\quad f\in C^{1,2}(\mathbb{R}^{+}\times\mathbb{R}^{d}).\label{eq:Generator}
\end{equation}
We used the short hand notation $\mathbb{P}^{s,x}$ and $\mathbb{E}^{s,x}$
for the associated probability measure and expectation respectively
for the process starting at $(s,x)\in\mathbb{R}^{+}\times\mathbb{R}^{d}$.
When a unique solution exists, one can define the Markov transition
kernel 
\begin{equation}
P(s,t,x,\Gamma):=\mathbb{P}^{s,x}(X_{t}\in\Gamma),\quad s<t,\Gamma\in\mathcal{B}.\label{eq:DefineMarkovTransitionByUniqueSoln}
\end{equation}
A unique solution exists when the Markov process $X_{t}$ is regular
i.e. for any $(s,x)\in\mathbb{R}^{+}\times\mathbb{R}^{d}$,
\begin{equation}
\mathbb{P}^{s,x}\{\tau=\infty\}=1,\label{eq:regular}
\end{equation}
where 
\[
\tau:=\lim_{n\rightarrow\infty}\tau_{n},\quad\tau_{n}:=\inf_{t\geq s}\{\lVert X_{t}\rVert\geq n\},\quad n\in\mathbb{N}.
\]

Using Proposition \ref{prop:PointwiseNondegDensity} and Theorem \ref{thm:DensityContinuityforLCD},
we give sufficient conditions in which the $T$-periodic Markov kernel
$P(s,s+T,\cdot,\cdot)$ of an SDE satisfies the local Doeblin condition
(\ref{eq:LocalDoeblinOtherTimeValues}). Specifically, we sufficiently
show (local) irreducibility and existence of a jointly continuous
density with respect to the Lebesgue measure $\Lambda$ (a Haar measure
on the $\mathbb{R}^{d}$, see Remark \ref{rem:HaarMeasure}).  It
is possible to show both properties simultaneously. For instance,
heat kernel estimates such as the classical one by Aronson \cite{Aronson_FundamentalBound}
sufficiently implies\textbf{ }the Proposition \ref{prop:PointwiseNondegDensity}
for non-autonomous SDEs with bounded drift and non-degenerate bounded
diffusion. 

Relaxing the non-degeneracy and boundedness assumption of \cite{Aronson_FundamentalBound},
it is well-known that autonomous SDEs satisfying (\ref{eq:regular})
and Hörmander's condition possesses a smooth density (globally with
respect to $\Lambda$) for the Markov transition kernel \cite{Malliavin_Hypo,HormanderBooks,RogersWilliamsVol2}.
However it is generally insufficient to yield irreducibility i.e.
Hörmander's condition does not imply the process can reach any given
non-empty open set with positive probability. We refer readers to
Remark 2.2 of \cite{Hairer_Hormander} for a counterexample. This
suggests some degree of non-degeneracy is required to imply irreducibility.
We emphasise that in existing literature, Hörmander's condition is
often applied for autonomous SDEs with relatively few existing results
for the non-autonomous case. Observe also that Theorem \ref{thm:DensityContinuityforLCD}
requires density of the transition kernel to exist locally rather
than globally. Recent advances by Höpfner, Löcherbach and Thieullen
gave the existence of a smooth local density of non-autonomous SDEs
under a time-dependent Hörmander's condition in \cite{HopfnerLocherbachThieullenTimeDepLocalHormanders}. 

Since the intention of this paper is to introduce main ideas and approach
to deduce the existence and uniqueness of periodic measures, we shall
show (global) irreducibility under the assumption that the diffusion
matrix and its inverse are bounded and utilise the results of \cite{HopfnerLocherbachThieullenTimeDepLocalHormanders}
for a local density. It will be the subject of future works to generalise
the results in the direction of local time-dependent Hörmander's condition
and relaxing the non-degeneracy assumption to attain a local irreducibility. 

Consider the following associated control system to (\ref{eq:NonAutonSDE})

\begin{equation}
\begin{cases}
dZ_{t}=\varphi(t)dt+\sigma(t,Z_{t})dW_{t}, & t\geq s,\\
Z_{s}=x,
\end{cases}\label{eq:Zt_ControlSystem}
\end{equation}
for some bounded adapted process $\varphi:\mathbb{R}^{+}\rightarrow\mathbb{R}^{d}$.
Inspired by the irreducibility argument of \cite{DaPratoZabczyk_ErgodicityInfDimensions},
we have the following lemma: 
\begin{lem}
\label{lem:Xt_Zt_equivalent(DPZ)}Assume $b$ and $\sigma$ are locally
Lipschitz and moreover $\sigma$ satisfies (\ref{eq:LinearGrowth_sigma})
and (\ref{eq:Sigma_and_inverse_bounded}). Assume further that there
exists a norm-like function $V$ and constant $c>0$ such that
\begin{equation}
\mathcal{L}(t)V\leq cV.\label{eq:RegularityCondition}
\end{equation}
Let $X_{t}=X_{t}^{s,x}$ and $Z_{t}=Z_{t}^{s,x}$ satisfy (\ref{eq:NonAutonSDE})
and (\ref{eq:Zt_ControlSystem}) respectively. Then the laws of $X_{t}$
and $Z_{t}$ are equivalent.
\end{lem}
\begin{proof}
By Theorem 3.5 of \cite{Hasminskii}, locally Lipschitz coefficients
and (\ref{eq:RegularityCondition}) yields that $X_{t}$ exists and
is unique. Since $\varphi$ is a bounded adapted process and $\sigma$
is locally Lipschitz with linear growth, by Theorem 3.1 of \cite{Mao_SDEs},
$Z_{t}$ also exists and is unique. Set $\tau_{n}=\inf_{t\geq s}\{\lVert Z_{t}\rVert\geq n\}$,
$Z_{t}^{n}=Z_{t\wedge\tau_{n}}$ and 
\[
\mathbb{P}^{n}(d\omega)=\mathbb{P}(d\omega)M_{t}^{n},
\]
where 
\[
M_{t}^{n}=\exp\left(-\frac{1}{2}\int_{s}^{t\wedge\tau_{n}}\alpha^{2}(r)dr-\int_{s}^{t\wedge\tau_{n}}\alpha(r)dW_{r}\right),
\]
and $\alpha(r)=\sigma^{-1}(r,Z_{r})[\varphi(r)-b(r,Z_{r})]$. It is
clear that $\alpha(r)$ is bounded for $s\leq r\leq\tau_{n}$, hence
Novikov condition is satisfied. Then Girsanov theorem implies
\[
\widetilde{W}_{t}^{n}=W_{t}+\int_{s}^{t}\alpha(r)dr
\]
is a Brownian motion on $\mathbb{R}^{d}$ under the probability measure
$\mathbb{P}^{n}$. It is clear that $d\widetilde{W}_{r}^{n}=dW_{r}+\alpha(r)dr$
and $\varphi(t)=\sigma(t,Z_{t})\alpha(t)+b(t,Z_{t})$ so
\begin{align*}
Z_{t}^{n} & =x+\int_{s}^{t\wedge\tau_{n}}\varphi(r)dr+\int_{s}^{t\wedge\tau_{n}}\sigma(r,Z_{r}^{n})dW_{r}\\
 & =x+\int_{s}^{t\wedge\tau_{n}}\left[\sigma(r,Z_{r}^{n})\alpha(r)+b(r,Z_{r}^{n})\right]dr+\int_{s}^{t\wedge\tau_{n}}\sigma(r,Z_{r}^{n})\left[d\widetilde{W}_{r}^{n}-\alpha(r)dr\right]\\
 & =x+\int_{s}^{t\wedge\tau_{n}}b(r,Z_{r}^{n})dr+\int_{s}^{t\wedge\tau_{n}}\sigma(r,Z_{r}^{n})d\widetilde{W}_{r}^{n}.
\end{align*}
i.e. $Z_{t}^{n}$ is a solution of (\ref{eq:NonAutonSDE}) on $(\Omega,\mathcal{F},\mathbb{P}^{n})$.
As the law of the solution does not depend on the choice of probability
space, we have that
\[
\mathbb{P}(X_{t}^{n}\in\Gamma)=\mathbb{P}^{n}(Z_{t}^{n}\in\Gamma),\quad\Gamma\in\mathcal{B}(\mathbb{R}^{d}).
\]
As $\mathbb{P}$ and $\mathbb{P}^{n}$ are equivalent, the laws of
$X_{t}^{n}$ and $Z_{t}^{n}$ are equivalent. This implies that
\begin{align*}
\mathbb{P}^{n}(\tau_{n}>t) & =\mathbb{P}^{n}\left(\sup_{s\le r\leq t}\lVert Z_{r}\rVert\leq n\right)=\mathbb{P}\left(\sup_{s\le r\leq t}\lVert X_{r}\rVert\leq n\right)\rightarrow1\quad\text{as }n\rightarrow\infty.
\end{align*}
 Define
\[
M_{t}=\exp\left(-\frac{1}{2}\int_{s}^{t}\alpha^{2}(r)dr-\int_{s}^{t}\alpha(r)dW_{r}\right).
\]
Then
\begin{align*}
\mathbb{E}[M_{t}] & \geq\mathbb{E}[M_{t}^{n}I_{\{\tau_{n}>t\}}]=\mathbb{P}^{n}(\tau_{n}>t)\rightarrow1\quad\text{as }n\rightarrow\infty.
\end{align*}
Moreover, we can prove that $\mathbb{P}(\tau_{n}>t)\rightarrow1$
as $n\rightarrow\infty$. This suggests from Borel-Cantelli Lemma
that there is a subsequence $n_{k}$ such that $\tau_{n_{k}}\rightarrow\infty$
almost surely where $n_{k}\rightarrow\infty$ as $k\rightarrow\infty$.
Thus
\[
M_{t}^{n_{k}}\rightarrow M_{t},\quad\text{as }k\rightarrow\infty
\]
almost surely. Now, by Fatou's lemma
\[
\lim_{k\rightarrow\infty}\mathbb{E}\left[M_{t}^{n_{k}}\right]\geq\mathbb{E}\left[\lim_{k\rightarrow\infty}M_{t}^{n_{k}}\right]=\mathbb{E}[M_{t}],
\]
and $\mathbb{E}[M_{t}^{n_{k}}]=1$ for each $k$ since $M_{t}^{n_{k}}$
is a martingale. So $\mathbb{E}[M_{t}]\leq1$. Thus, we have that
$\mathbb{E}[M_{t}]=1$. Now we apply Girsanov theorem \cite{DaPratoZabcyzk92}
to yield that 
\[
\widetilde{W}_{t}=W_{t}+\int_{s}^{t}\alpha(r)dr
\]
is a Brownian motion on $\mathbb{R}^{d}$ under the probability measure
$\widetilde{\mathbb{P}}$ where $\widetilde{\mathbb{P}}(d\omega)=\mathbb{P}(d\omega)M_{t}$.
As before, 
\begin{align*}
Z_{t} & =x+\int_{s}^{t}\varphi(r)dr+\int_{s}^{t}\sigma(r,Z_{r})dW_{r}=x+\int_{s}^{t}b(r,Z_{r})dr+\int_{s}^{t}\sigma(r,Z_{r})d\widetilde{W}_{r}
\end{align*}
is a solution to (\ref{eq:NonAutonSDE}) on $(\Omega,\mathcal{F},\widetilde{\mathbb{P}})$.
As the law of the solution does not depend on the choice of probability
space, we have that
\[
\mathbb{P}(X_{t}\in\Gamma)=\widetilde{\mathbb{P}}(Z_{t}\in\Gamma),\quad\Gamma\in\mathcal{B}(\mathbb{R}^{d}).
\]
As $\mathbb{P}$ and $\widetilde{\mathbb{P}}$ are equivalent, the
laws of $X_{t}$ and $Z_{t}$ are equivalent.
\end{proof}
\begin{thm}
\label{thm:DPZIrred}Consider SDE (\ref{eq:NonAutonSDE}) and assume
the same conditions as Lemma \ref{lem:Xt_Zt_equivalent(DPZ)}. Then
the Markov transition kernel $P(s,t,\cdot,\cdot)$ for $s<t<\infty$
is irreducible i.e. $P(s,t,x,\Gamma)>0$ for all $x\in\mathbb{R}^{d}$
and non-empty open $\Gamma\in\mathcal{B}(\mathbb{R}^{d})$.
\end{thm}
\begin{proof}
By Lemma \ref{lem:Xt_Zt_equivalent(DPZ)}, as $\mathbb{P}$ and $\widetilde{\mathbb{P}}$
are equivalent,\textbf{ }it is sufficient to show that for any $\delta>0$
and $x,a\in\mathbb{R}^{d}$ that
\[
\mathbb{P}(\lVert Z_{t}^{s,x}-a\rVert<\delta)>0.
\]
We consider the auxiliary system
\begin{equation}
\begin{cases}
dY_{t}=\sigma(t,Y_{t})dW_{t},\\
Y_{s}=x.
\end{cases}\label{eq:AuxillarySystemY}
\end{equation}
Since $\sigma$ is Lipschitz, then (\ref{eq:AuxillarySystemY}) has
a unique solution $Y_{t}$ satisfying
\begin{equation}
Y_{t}=x+\int_{s}^{t}\sigma(r,Y_{r})dW_{r}.\label{eq:RepresentationofY}
\end{equation}
For $u\in[s,t)$, $R>0$ and $\widetilde{a}\in\mathbb{R}^{d}$ all
to be chosen later, pick a bounded function $f:[u,t]\times\mathbb{R}^{d}\rightarrow\mathbb{R}^{d}$
such that $f$ is Lipschitz and 
\[
f(r,y)=\begin{cases}
0 & \text{if }\lVert y\rVert>2R,\\
\frac{\widetilde{a}-y}{t-u} & \text{if }\lVert y\rVert\leq R.
\end{cases}
\]
Define the integral
\[
I_{1}(y)=y+\int_{u}^{t}f(r,y)dr,\quad y\in\mathbb{R}^{d}.
\]
Observe that if $\lVert y\rVert\leq R$,
\begin{align}
I_{1}(y) & =y+\frac{1}{t-u}\int_{u}^{t}(\widetilde{a}-y)dr=\widetilde{a}.\label{eq:SemigroupIdentity}
\end{align}
Set 
\[
\varphi(r)=\begin{cases}
0 & \text{if }r\in[s,u),\\
f(r,Y_{u}) & \text{if }r\in[u,t].
\end{cases}
\]
Then it is clear that $Z_{r}^{s,x}=Y_{r}$ for $r\in[s,u)$. Hence,
by sample-path continuity of $Y_{t}$, $Z_{t}$ can be represented
as an initial-valued SDE in terms of $Y_{u}$ namely 
\[
Z_{t}^{s,x}=Y_{u}+\int_{u}^{t}f(r,Y_{u})dr+\int_{u}^{t}\sigma(r,Z_{r})dW_{r}.
\]
Let $I_{1}=I_{1}(Y_{u})$ and $I_{2}=\int_{u}^{t}\sigma(r,Z_{r})dW_{r}.$
Then $Z_{t}^{s,x}=I_{1}+I_{2}$. Choose any fixed $\widetilde{a}\in\mathbb{R}^{d}$
such that
\[
\lVert a-\widetilde{a}\rVert\leq\frac{\delta}{3}.
\]
Suppose the events $\{I_{1}=\widetilde{a}\}$ and $\{\lVert I_{2}\rVert\leq\frac{\delta}{3}\}$
holds then 
\[
\lVert Z_{t}^{s,x}-a\rVert=\lVert\left(I_{1}-\widetilde{a}\right)+(I_{2}+\widetilde{a}-a)\rVert\leq\lVert I_{2}\rVert+\lVert\widetilde{a}-a\rVert\leq\frac{2}{3}\delta.
\]
Hence 
\begin{equation}
P(\lVert Z_{t}^{s,x}-a\rVert\leq\delta)\geq P(I_{1}=\widetilde{a}\text{ and }\lVert I_{2}\rVert\leq\frac{\delta}{3})\geq\mathbb{P}\left(I_{1}=\widetilde{a}\right)-\mathbb{P}\left(\lVert I_{2}\rVert>\frac{\delta}{3}\right),\label{eq:P(Zt_less_than_delta)}
\end{equation}
where we used the elementary inequality $\mathbb{P}(A\cap B)\geq\mathbb{P}(A)-\mathbb{P}(B^{c})$
for any event $A,B\in\mathcal{F}$. Thus the proof is complete provided
the right hand side of inequality (\ref{eq:P(Zt_less_than_delta)})
is positive. By Chebyshev\textquoteright s inequality and Itô's isometry,
\begin{align*}
\mathbb{P}\left(\lVert I_{2}\rVert>\frac{\delta}{3}\right) & \leq\frac{9}{\delta^{2}}\mathbb{E}[\lVert I_{2}\rVert^{2}]\leq\frac{9}{\delta^{2}}\int_{u}^{t}\lVert\sigma(r,Z_{r})\rVert^{2}dr\leq\frac{9}{\delta^{2}}(t-u)\lVert\sigma\rVert_{\infty}^{2}.
\end{align*}
Hence, one can fix a $u\in[s,t)$ such that 
\[
\mathbb{P}\left(\lVert I_{2}\rVert>\frac{\delta}{3}\right)\leq\frac{1}{4}.
\]
Similarly, for the fixed $u$ and any $R>0$, from (\ref{eq:RepresentationofY})
we have
\begin{align*}
\mathbb{P}(\lVert Y_{u}\rVert & >R)\leq\frac{1}{R^{2}}\mathbb{E}[\lVert Y_{u}\rVert^{2}]=\frac{1}{R^{2}}\left(\lVert x\rVert^{2}+\lVert\int_{u}^{s}\sigma(r,Y_{r})dr\rVert^{2}\right)\leq\frac{1}{R^{2}}\left[\lVert x\rVert^{2}+\lVert\sigma\rVert_{\infty}^{2}(u-s)^{2}\right].
\end{align*}
Hence one can fix a sufficiently large $R>0$ such that
\begin{equation}
\mathbb{P}(\lVert Y_{u}\rVert\leq R)\geq\frac{3}{4}.\label{eq:ProbYLessThanR}
\end{equation}
By (\ref{eq:SemigroupIdentity}), we have the inclusion $\{\lVert Y_{u}\rVert\leq R\}\subset\{I_{1}=\widetilde{a}\}$.
Hence by (\ref{eq:ProbYLessThanR})
\[
\mathbb{P}(I_{1}=\widetilde{a})\geq\mathbb{P}(\lVert Y_{u}\rVert\leq R)\geq\frac{3}{4}.
\]
The proof is complete by the following inequality for irreducibility
\[
\mathbb{P}(\lVert Z_{t}^{s,x}-a\rVert\leq\delta)=\mathbb{P}\left(I_{1}=\widetilde{a}\right)-\mathbb{P}\left(\lVert I_{2}\rVert>\frac{\delta}{3}\right)\geq\frac{3}{4}-\frac{1}{4}=\frac{1}{2}.
\]
\end{proof}
In the next theorem, we apply Theorem 1 of \cite{HopfnerLocherbachThieullenTimeDepLocalHormanders}
to attain a smooth density of transition probabilities in extension
of classical results by Aronson \cite{Aronson_FundamentalBound} for
parabolic equations with bounded time-dependent coefficients. We assume
that $\sigma$ is time-independent as in \cite{HopfnerLocherbachThieullenTimeDepLocalHormanders}.
It would be of future works to study the possible generalisation of
Theorem 1 of \cite{HopfnerLocherbachThieullenTimeDepLocalHormanders}
for $T$-periodic $\sigma$.
\begin{thm}
\label{thm:SDE_limiting_PM}Consider SDE (\ref{eq:NonAutonSDE}) and
assume the same conditions as Lemma \ref{lem:Xt_Zt_equivalent(DPZ)}.
Assume that (\ref{eq:LocallySmoothSigma}) and (\ref{eq:LocallySmoothDrift})
holds. Assume further that there exists a compact set $K\in\mathcal{B}(\mathbb{R}^{d})$
such that (\ref{eq:FosterLya1_s*}) and (\ref{eq:FosterLya2_s*})
hold. Then the results of Theorem \ref{thm:Limiting_PM_on_E} hold.
\end{thm}
\begin{proof}
The invertibility of $\sigma$ implies linear independent columns
hence our collective assumptions satisfy Theorem 1 of \cite{HopfnerLocherbachThieullenTimeDepLocalHormanders}.
Hence there exists a smooth density $p(s,t,x,y)$ with respect to
$\Lambda$. Then using Theorem \ref{thm:DPZIrred}, we have that Theorem
\ref{thm:DensityContinuityforLCD} holds. Hence the assumptions of
Theorem \ref{thm:Limiting_PM_on_E} are satisfied. 
\end{proof}

\subsection{Geometric Ergodicity of Periodic Measures}

In the previous section, we studied limiting periodic measures in
a qualitative manner. We extend this for geometrically ergodic periodic
measures. That is, the convergence towards the periodic measure is
exponentially fast. We recall the geometric drift condition for SDE.
\begin{defn}
The SDE (\ref{eq:NonAutonSDE}) is said to satisfy the geometric drift
condition if there exists a function $V\in C^{1,2}(\mathbb{R}^{+}\times\mathbb{R}^{d},\mathbb{R}^{+})$
and constants $C\geq0$ and $\lambda>0$ such that 
\begin{equation}
\mathcal{L}(t)V\leq C-\lambda V\quad\text{on }\mathbb{R}^{+}\times\mathbb{R}^{d},\label{eq:GeometricDriftCdn}
\end{equation}
where $\mathcal{L}(t)$ is given by (\ref{eq:Generator}). 
\end{defn}
Note that if (\ref{eq:GeometricDriftCdn}) is satisfied then the SDE
is regular. Specifically, since $V\geq0$ and $\mathcal{L}(t)[\text{const}]=0$,
it is easy to see that
\[
\mathcal{L}(t)(V+1)\leq C-\lambda V\leq C\leq C(V+1),
\]
hence the regularity condition (\ref{eq:RegularityCondition}) is
satisfied. 

Using the geometric drift condition, we give one of the main results
on the existence, uniqueness and geometric ergodicity of a periodic
measure. It is worth noting that if the SDE coefficients have a trivial
period, then the theorem recovers known results of invariant measures.
Hence, the results here presented can be regarded as time-periodic
generalisations of such theorems of invariant measures for autonomous
SDEs. 
\begin{thm}
\label{thm:GeometricConvergence-PeriodicMeasure}Assume $T$-periodic
SDE (\ref{eq:NonAutonSDE}) coefficients satisfies (\ref{eq:LinearGrowth_sigma}),
(\ref{eq:Sigma_and_inverse_bounded}), (\ref{eq:LocallySmoothSigma})
and (\ref{eq:LocallySmoothDrift}). Assume further that there exists
a $T$-periodic norm-like $V\in C^{1,2}(\mathbb{R}^{+}\times\mathbb{R}^{d},\mathbb{R}^{+})$
satisfying the geometric drift condition (\ref{eq:GeometricDriftCdn}).
Then Theorem \ref{thm:Geometric_PM_on_E} follows.
\end{thm}
\begin{proof}
By Itô's formula and the regularity of $V$, one has
\[
d\left(e^{\lambda t}V(t,X_{t})\right)=e^{\lambda t}(\lambda V+\mathcal{L}(t)V)dt+e^{\lambda t}\langle\nabla V,\sigma dW_{t}\rangle,
\]
hence by the geometric drift condition
\begin{align}
V(t,X_{t}) & =e^{-\lambda(t-s)}V(s,X_{s})+\int_{s}^{t}e^{-\lambda(t-r)}(\lambda V+\mathcal{L}(r)V)dr+\int_{s}^{t}e^{-\lambda(t-r)}\langle\nabla V,\sigma dW_{r}\rangle\nonumber \\
 & \leq e^{-\lambda(t-s)}V(s,X_{s})+\frac{C}{\lambda}\left(1-e^{-\lambda(t-s)}\right)+\int_{s}^{t}e^{-\lambda(t-r)}\langle\nabla V,\sigma dW_{r}\rangle.\label{eq:d(expV)}
\end{align}
By (\ref{eq:regular}) and the regularity of $V$, $\int_{s}^{t}e^{\lambda(t-r)}\langle\sigma^{T}(r,X_{r})\nabla V(r,X_{r}),dW_{r}\rangle_{\mathbb{R}^{d}}$
is a martingale. Hence 
\begin{equation}
\mathbb{E}^{s,x}[V(t,X_{t})]\leq e^{-\lambda(t-s)}V(s,x)+\frac{C}{\lambda}(1-e^{-\lambda(t-s)})\quad s\leq t.\label{eq:ExpectedVAnyTimes}
\end{equation}
Specifically,
\begin{equation}
\mathbb{E}^{s,x}[V(s+T,X_{s+T})]\leq e^{-\lambda T}V(s,x)+\frac{C}{\lambda}(1-e^{-\lambda T}).\label{eq:ExpectedV}
\end{equation}
Define the functions $U_{s}(\cdot):=V(s,\cdot)\geq0$. Since $V$
is $T$-periodic, we have that (\ref{eq:ExpectedV}) is equivalent
to 
\begin{align}
P(s,s+T)U_{s}(x) & \leq e^{-\lambda T}U_{s}(x)+\frac{C}{\lambda}(1-e^{-\lambda T})\label{eq:ExpDecay-OneStep}
\end{align}
That is to say (\ref{eq:MarkovChainGeometricDriftCondition}) is satisfied
for each $s\geq0$. Subtracting $U_{s}(x)$ from (\ref{eq:ExpDecay-OneStep})
yields
\begin{align*}
P(s,s+T)U_{s}(x)-U_{s}(x) & \leq(1-e^{-\lambda T})\left(\frac{C}{\lambda}-U_{s}(x)\right).
\end{align*}
Since $U_{s}$ is norm-like assumption, define for $\epsilon>0$
\[
K=\bigcap_{s\in[0,T]}K_{s},\quad\text{where }K_{s}:=\left\{ x\in\mathbb{R}^{d}|U_{s}(x)\leq\frac{C}{\lambda}+\frac{\epsilon}{1-e^{-\lambda T}}\right\} .
\]
For sufficiently large $\epsilon$, $K$ is non-empty compact set.
Since the SDE is regular, the same proof from Theorem \ref{thm:SDE_limiting_PM}
implies that Theorem \ref{thm:DensityContinuityforLCD} holds i.e.
$P(s,s+T,x,\cdot)$ satisfies the local Doeblin condition for each
$s\geq0$. Thus the conditions of Theorem \ref{thm:Geometric_PM_on_E}
are met.
\end{proof}
Theorem \ref{thm:GeometricConvergence-PeriodicMeasure} depends crucially
on finding a suitable Foster-Lyapunov function $V$. Dissipative SDEs
are special cases where the Euclidean norm is a such Foster-Lyapunov
function. This has the advantage that it can be simpler to verify
that the geometric drift condition is satisfied. The definition of
dissipativity below coincides with that of Hale \cite{HaleDissip}
when the SDE is deterministic ($\sigma=0$). 
\begin{defn}
SDE (\ref{eq:NonAutonSDE}) is weakly dissipative if there exists
constants $c,\lambda>0$ such that 
\begin{equation}
2\langle b(t,x),x\rangle\leq c-\lambda\lVert x\rVert^{2}\quad\text{on }\mathbb{R}^{+}\times\mathbb{R}^{d},\label{eq:Dissip}
\end{equation}
and dissipative if $c=0$.
\end{defn}
\begin{cor}
\label{Cor:Dissip} 
\label{zhao2019.4a}
Assume $T$-periodic SDE (\ref{eq:NonAutonSDE})
coefficients satisfies (\ref{eq:LinearGrowth_sigma}), (\ref{eq:Sigma_and_inverse_bounded}),
(\ref{eq:LocallySmoothSigma}) and (\ref{eq:LocallySmoothDrift})
and is weakly dissipative. Then Theorem \ref{thm:Geometric_PM_on_E}
holds.
\end{cor}
\begin{proof}
By Theorem \ref{thm:GeometricConvergence-PeriodicMeasure}, it suffices
to show $V(t,x)=\lVert x\rVert^{2}$ satisfies the geometric drift
condition. We compute that 
\begin{align*}
\mathcal{L}(t) & \lVert x\rVert^{2}=\langle b(t,x),x\rangle+\sum_{i,j}^{d}a_{ii}(t,x)\leq c-\lambda\lVert x\rVert^{2}+\lVert\sigma\rVert_{\infty}.
\end{align*}
i.e. $\lVert x\rVert^{2}$ satisfies the geometric drift condition
with $C=c+\lVert\sigma\rVert_{\infty}$ and same $\lambda$ from (\ref{eq:Dissip}). 
\end{proof}
\begin{thm}
\label{thm:OddPolyDrift}Consider $T$-periodic SDE (\ref{eq:NonAutonSDE})
with $\sigma$ satisfying (\ref{eq:LinearGrowth_sigma}), (\ref{eq:Sigma_and_inverse_bounded})
and (\ref{eq:LocallySmoothSigma}) and drift 
\[
b(t,x)=\left(\begin{array}{c}
\sum_{k=0}^{2p_{1}-1}S_{k}^{1}(t)x_{1}^{k}\\
\vdots\\
\sum_{k=0}^{2p_{d}-1}S_{k}^{d}(t)x_{d}^{k}
\end{array}\right),
\]
where $\{p_{i}\}_{i=1}^{d}\in\mathbb{N}\backslash\{0\}$, $\{S_{k}^{i}\}_{i=1...d}^{k=1...2p_{i}-2}$
are continuously differentiable $T$-periodic functions and constants
$S_{2p_{i}-1}^{i}<0$ . Then Theorem \ref{thm:GeometricConvergence-PeriodicMeasure}
holds.
\end{thm}
\begin{proof}
Clearly $b$ satisfies (\ref{eq:LocallySmoothDrift}). Hence by Corollary
\ref{Cor:Dissip}, it suffices to show that the SDE is weakly dissipative.
We compute that 
\[
\langle b(t,x),x\rangle=\sum_{i=1}^{d}\sum_{k=1}^{2p_{i}}S_{k}^{i}x_{i}^{k}.
\]
For each fixed $1\leq i\leq d$, $\sum_{k=1}^{2p_{i}}S_{k}^{i}x_{i}^{k}$
is an even degree polynomial with leading negative coefficient. By
assumption, $\{S_{k}^{i}\}$ are all bounded hence, fixing a $\lambda\in(0,-\min_{i}S_{2p_{i}-1}^{i})$,
define the constants
\[
\widetilde{c}_{i}:=\sup_{x_{i}\in\mathbb{\mathbb{R}},t\in[0,T]}\left(\sum_{k=1}^{2p_{i}}S_{k}^{i}x_{i}^{k}+\lambda x_{i}^{2p_{i}}\right)<\infty,\quad c_{i}:=\widetilde{c}_{i}+\sup_{x_{i}\in\mathbb{R}}\lambda\left(x_{i}^{2}-x_{i}^{2p_{i}}\right)<\infty,
\]
then we deduce the SDE is weakly dissipative by 
\begin{align*}
\langle b(t,x),x\rangle & \leq\sum_{i=1}^{d}\left(\widetilde{c}_{i}-\lambda x_{i}^{2p_{i}}\right)\leq\sum_{i=1}^{d}\left(c_{i}-\lambda x_{i}^{2}\right)=\sum_{i=1}^{d}c_{i}-\lambda\lVert x\rVert^{2}.
\end{align*}
\end{proof}
As it would be more apparent in the next section of gradient SDEs,
Theorem \ref{thm:OddPolyDrift} has many physical applications. They
model multi-stable systems such as modulated Josephson-junctions systems,
superionic conductors, excited chicken hearts to the dithered ring
lasers as well as other laser systems. We refer to \cite{ZhouMossJung}
and references therein for further details of these applications. 

We give two specific examples of Theorem \ref{thm:OddPolyDrift}.
First, we consider periodically forced mean-reverting Ornstein-Uhlenbeck
processes. In this example, we compute the density of the process,
periodic measure and the exponential convergence rate explicitly.
While the computations are straightforward, it appears that the periodic
measure and its geometric convergence for this system has not been
previously noted in literature. The classical Ornstein-Uhlenbeck process
is mean-reverting and has a geometric invariant measure. This contrasts
with periodically forced Ornstein-Uhlenbeck processes which does not
have limiting invariant measure. Instead, the system has a limiting
periodic measure and mean-reverting to a periodic mean. We expect
this to be useful for systems possessing periodic mean reversion due
to factors such as seasonality. 
\begin{example}
\label{exa:PeriodicOUExample}Consider the following multidimensional
Ornstein-Uhlenbeck equation 
\begin{equation}
dX_{t}=\left(S(t)-AX_{t}\right)dt+\sigma dW_{t},\label{eq:Multidimensional_OU}
\end{equation}
where $A=M^{-1}DM\in\mathbb{R}^{d\times d}$ for some $M\in GL(\mathbb{R}^{d})$
and $D=\text{diag}(\lambda_{1},\cdot\cdot\cdot,\lambda_{d})$ is a
diagonal matrix with positive eigenvalues $\{\lambda_{n}\}_{n=1}^{d}$,
$\sigma\in GL(\mathbb{R}^{d})$ and $S(t):\mathbb{R}^{+}\rightarrow\mathbb{R}^{d}$
be a $T$-periodic continuously differentiable function. By applying
Itô's formula on $e^{tA}X_{t}$ or by a variation of constants formula,
we have 
\begin{equation}
X_{t}=e^{-(t-s)A}X_{s}+\int_{s}^{t}e^{-(t-r)A}S(r)dr+\int_{s}^{t}e^{-(t-r)A}\sigma dW_{r}\quad t\geq s.\label{eq:VariationOfConst}
\end{equation}
Observe that $\xi(t):=\int_{-\infty}^{t}e^{-A(t-r)}S(r)dr$ satisfies
$\partial_{r}(e^{Ar}\xi)=e^{Ar}S(r)$ and is continuous and $T$-periodic.
Then
\begin{align}
J(s,t): & =\int_{s}^{t}e^{-(t-r)A}S(r)dr=\xi(t)-e^{-(t-s)A}\xi(s).\label{eq:J(s,t)}
\end{align}
By the $T$-periodicity of $\xi$, it is clear $\lim_{t\rightarrow\infty}J(s,t)$
does not converge. Instead, it converges along integral multiples
of the period in the following way: let $Id$ be the identity matrix
on $\mathbb{R}^{d}$ and define
\begin{align*}
\xi_{n}(s) & :=J(s,s+nT)=(Id-e^{-nTA})\xi(s),\quad n\in\mathbb{N},
\end{align*}
Then $\xi(s)=\lim_{n\rightarrow\infty}\xi_{n}(s)$. We shall see $\xi(s)$
as the ``long term periodic mean''. From (\ref{eq:VariationOfConst}),
it is easy to see that $X_{t}$ is normally distributed. Specifically,
we can compute
\[
\mathbb{E}^{s,x}[X_{t}]=e^{-(t-s)A}x+J(s,t).
\]
Since $A=M^{-1}DM$ then $e^{-(t-r)A}=M^{-1}e^{-(t-r)D}M$, thus denoting
$N=M\sigma$, component-wise, we have $(e^{-(t-r)D}NdW_{r})_{i}=e^{-(t-r)\lambda_{i}}\sum_{k=1}^{d}N_{ik}dW_{r}^{k}$.
Hence by independence of Brownian motion and properties of Itô's inner-product,
we have
\begin{align*}
C_{ij}(s,t) & :=\mathbb{E}^{s,x}\left[\left(\int_{s}^{t}(e^{-(t-r)D}NdW_{r})_{i}\right)\left(\int_{s}^{t}(e^{-(t-r)D}NdW_{r})_{j}\right)\right]\\
 & =\sum_{k,k'=1}^{d}N_{ik}N_{jk'}\mathbb{E}^{s,x}\left[\left(\int_{s}^{t}e^{-(t-r)\lambda_{i}}dW_{r}^{k}\right)\left(\int_{s}^{t}e^{-(t-r)\lambda_{j}}dW_{r}^{k'}\right)\right]\\
 & =\sum_{k=1}^{d}N_{ik}N_{jk}\mathbb{E}^{s,x}\left[\int_{s}^{t}e^{-(t-r)(\lambda_{i}+\lambda_{j})}dr\right]\\
 & =\frac{(M\sigma\sigma^{T}M^{T})_{ij}}{\lambda_{i}+\lambda_{j}}\left(1-e^{-(t-s)(\lambda_{i}+\lambda_{j})}\right).
\end{align*}
Hence the covariance matrix $\text{Cov}(X_{t}|X_{s}=x):=\mathbb{E}^{s,x}[X_{t}X_{t}^{T}]-\mathbb{E}^{s,x}[X_{t}]\mathbb{E}^{s,x}[X_{t}^{T}]=M^{-1}C(s,t)M$,
where $C(s,t)$ has entries $C_{ij}(s,t)$ as defined above. Thus,
denoting $\mathcal{N}$ for the multivariate normal distribution,
the Markov transition kernel of (\ref{eq:Multidimensional_OU}) is
given by
\begin{equation}
P(s,t,x,\cdot)=\mathcal{N}\left(e^{-(t-s)A}x+J(s,t),M^{-1}C(s,t)M\right)(\cdot),\label{eq:LawOfPeriodicMultidimensionalOU}
\end{equation}
\end{example}
Since $\lim_{t\rightarrow\infty}J(s,t)$ does not converges (for any
fixed $s$), (\ref{eq:LawOfPeriodicMultidimensionalOU}) does not
converge. This implies there does not exist a limiting invariant measure
for this periodically forced Ornstein-Uhlenbeck process. This contrasts
with the classical Ornstein-Uhlenbeck process (when $S(t)=\text{const})$,
where one often take $t\rightarrow\infty$ to yield a (unique) limiting
invariant measure. On the other hand, for every fixed $s$, along
integral multiple of the period i.e. $t=s+nT$, one has directly from
(\ref{eq:LawOfPeriodicMultidimensionalOU})

\begin{align}
P(s,s+nT,x,\cdot) & =\mathcal{N}\left(e^{-nTA}x+\xi_{n}(s),M^{-1}C(s,s+nT)M\right)(\cdot)\nonumber \\
 & \rightarrow\mathcal{N}\left(\xi(s),M^{-1}CM\right)(\cdot)=:\rho_{s}(\cdot),\label{eq:PM_OU}
\end{align}
as $n\rightarrow\infty$ where $C$ is the matrix with entries $C_{ij}=\frac{(M\sigma\sigma^{T}M^{T})_{ij}}{\lambda_{i}+\lambda_{j}}$.
That is to say that the long-time behaviour is characterised by $\rho_{s}$
for every fixed $s\geq0$. Since $\xi$ is $T$-periodic, $\rho$
is also $T$-periodic. Moreover, it easy to explicitly verified that
$\rho$ is periodic measure of the system. It is worth noting that
for every $s$, $\text{supp}(\rho_{s})=\mathbb{R}^{d}$ and that the
periodic measure (\ref{eq:PM_OU}) is unique. This contrasts with
the time-homogeneous Markovian systems, where the uniqueness of periodic
measure (if exist) holds only if it is supported by disjoint Poincaré
sections \cite{CFHZ2016}. 

The above calculations gives the existence and uniqueness of a periodic
measure. However, it does not immediately give a convergence rate.
For simplicity, we show the convergence and its rate for the one-dimensional
case. We first recall that the Kullback-Leibler divergence, $D_{KL}(\cdot||\cdot)$,
is pre-metric on $\mathcal{P}(\mathbb{R}^{d})$. Let $P,Q\in\mathcal{P}(\mathbb{R}^{d})$
with densities $p,q\in L^{1}(\mathbb{R}^{d})$ respective, the Kullback-Leibler
divergence can be defined by
\begin{align*}
D_{KL}(P||Q): & =\int_{\mathbb{R}^{d}}p(x)\log\left(\frac{p(x)}{q(x)}\right)dx.
\end{align*}
For vectors $\mu_{i}\in\mathbb{R}^{d}$ and matrices $\sigma_{i}\in GL(\mathbb{R}^{d})$
where $i=1,2$, we have specifically the following explicit expression
for normal densities. 
\[
D_{KL}(\mathcal{N}(\mu_{1},\sigma_{1})||\mathcal{N}(\mu_{2},\sigma_{2}))=\frac{1}{2}\left(\text{Tr}(\sigma_{2}^{-1}\sigma_{1})+(\mu_{2}-\mu_{1})^{T}\sigma_{2}^{-1}(\mu_{2}-\mu_{1})-d+\ln\left(\frac{\det(\sigma_{2})}{\det(\sigma_{1})}\right)\right).
\]
Moreover, Pinsker's inequality states 
\[
\left\lVert P-Q\right\rVert _{TV}^{2}\leq\frac{1}{2}D_{KL}(P||Q),\quad P,Q\in\mathcal{P}(E).
\]
We recall the elementary identity $\ln(1-y)=-\sum_{k=1}^{\infty}\frac{y^{k}}{k}$
for any fixed $y\in(-1,1)$. Hence, the following elementary inequality
holds by a geometric sum
\begin{align*}
-(y+\ln(1-y)) & =\sum_{k=2}^{\infty}\frac{y^{k}}{k}\leq\frac{y}{2}\sum_{k=1}^{\infty}y^{k}\leq\frac{y}{2}\frac{1}{1-y},\quad y\in(0,1).
\end{align*}
Now, since both $\rho_{s+t}$ and $P(s,s,+t,x,\cdot)$ are normally
distributed, by Pinsker's inequality and (\ref{eq:J(s,t)}), for all
$t\geq\delta$ and for any fixed $\delta>0$, $x\in\mathbb{R}^{d}$
\begin{align*}
\left\lVert P(s,s+t,x,\cdot)-\rho_{s+t}\right\rVert _{TV}^{2} & \leq\frac{1}{2}D_{KL}(P(s,s+t,x,\cdot)||\rho_{s+t})\\
 & =\frac{1}{4}\left(1-e^{-2tA}+\frac{(\xi(s+t)-e^{-tA}x-J(s,s+t))^{2}}{\sigma^{2}/2\alpha}-1-\ln(1-e^{-2tA})\right)\\
 & =\frac{1}{4}\left(\frac{e^{-2tA}(\xi(s)-x)^{2}}{\sigma^{2}/2\alpha}+\frac{1}{2}\frac{e^{-2tA}}{1-e^{-2tA}}\right)\\
 & \leq\frac{e^{-2tA}}{\sigma^{2}}\frac{A}{2}\left((\xi(s)-x)^{2}+\frac{\sigma^{2}}{4A}\frac{1}{1-e^{-2\delta A}}\right).
\end{align*}
Deducing indeed the convergence is geometric. We go a little further
solely to align with Theorem \ref{thm:GeometricConvergence-PeriodicMeasure}.
For every fixed $s\in[0,T)$ and for any fixed $\gamma>0$, there
exists a constant $r_{s}=r_{s}(\gamma)>0$ such that 
\[
(x-\xi(s))^{2}-(1+\gamma)x^{2}=-2\xi(s)x+\xi^{2}(s)-\gamma x^{2}\leq r_{s}.
\]
Hence $(x-\xi(s))^{2}\leq(1+\gamma)x^{2}+r_{s}$. Define $R_{s}:=\max\left\{ 1+\gamma,r_{s}+\frac{\sigma^{2}}{4A}\frac{1}{1-e^{-2A\delta}}\right\} >1$,
then 
\begin{align*}
\left\lVert P(s,s+t,x,\cdot)-\rho_{s+t}\right\rVert _{TV}^{2} & \leq\frac{e^{-2At}}{\sigma^{2}}\frac{A}{2}R_{s}\left(x^{2}+1\right)\leq e^{-2At}\left(\sqrt{\frac{\alpha}{2}}\frac{R_{s}}{\sigma}\right)^{2}\left(x^{2}+1\right)^{2},
\end{align*}
where we trivially squared the last two terms. Specifically by letting
$t=nT$, we have geometric ergodicity of the grid chain 
\[
\left\lVert P(s,s+nT,x,\cdot)-\rho_{s}\right\rVert _{TV}^{2}\leq e^{-2nTA}\left(\sqrt{\frac{A}{2}}\frac{R_{s}}{\sigma}\right)^{2}\left(x^{2}+1\right)^{2},\quad n\in\mathbb{N}.
\]
For computationally inclined readers, we give explicit formula for
$\xi_{t}$ in the one dimensional case. Multidimensional case can
be computed similarly. By Fourier Series, for any $S\in L^{2}[0,T]$,
$S$ can be represented by
\[
S(t)=\frac{A_{0}}{2}+\sum_{n=1}^{\infty}A_{n}\cos\left(\frac{2n\pi}{T}t-n\pi\right)+B_{n}\sin\left(\frac{2n\pi}{T}t-n\pi\right),
\]
with the usual Fourier coefficients for $n\in\mathbb{N}\backslash\{0\}$
\[
A_{n}=\frac{2}{T}\int_{0}^{T}S(t)\cos\left(\frac{2n\pi}{T}t-n\pi\right)dt,\quad B_{n}=\frac{2}{T}\int_{0}^{T}S(t)\sin\left(\frac{2n\pi}{T}t-n\pi\right)dt.
\]
It is trivial to see $\xi_{0}:=\frac{1}{A}\frac{A_{0}}{2}$ satisfies
$\partial_{t}(e^{At}\xi_{0})=\frac{A_{0}}{2}e^{tA}$. Similarly, 
\[
\xi_{n}^{cos}(t):=\frac{1}{A}\frac{T^{2}\cos\left(n\pi-\frac{2n\pi}{T}t\right)-2n\pi T\sin\left(n\pi-\frac{2n\pi}{T}t\right)}{4\pi^{2}n^{2}+T^{2}}
\]
satisfies $\partial_{t}(e^{tA}\xi_{n}^{cos}(t))=e^{tA}\cos\left(\frac{2n\pi}{T}t-n\pi\right)$
and 
\[
\xi_{n}^{sin}(t):=-\frac{1}{A}\frac{T^{2}\sin\left(n\pi-\frac{2n\pi}{T}t\right)+2n\pi T\cos\left(n\pi-\frac{2n\pi}{T}t\right)}{4\pi^{2}n^{2}+T^{2}}
\]
satisfies $\partial_{t}(e^{tA}\xi_{n}^{sin}(t))=e^{tA}\sin\left(\frac{2n\pi}{T}T-n\pi\right)$.
Clearly $\xi_{n}^{cos}$ and $\xi_{n}^{sin}$ are both $T$-periodic
and $\xi(t):=\xi_{0}+\sum_{i=1}^{\infty}A_{n}\xi_{n}^{cos}(t)+B_{n}\xi_{n}^{sin}(t)$
is the desired $T$-periodic continuous (hence) bounded function satisfying
$\partial_{t}(e^{tA}\xi)=e^{tA}S$. 
\begin{example}
\label{Exa:DuffingOsc} The stochastic overdamped Duffing Oscillator
has many physical applications including the phenomena of stochastic
resonance in climate dynamics modelling of ice age \cite{BenziStochRes,NicolisPeriodicForcing,JungPerioidcallyDrivenSys}
as mentioned in details in the introduction. We also expect applications
to periodically-forced model of price dynamics in the financial markets
akin to \cite{Lima_Miranda_Econ} with a similar interpretation. The
Duffing Oscillator is given by
\begin{equation}
dX_{t}=\left[-X_{t}^{3}+X_{t}+A\cos(\omega t)\right]dt+\sigma dW_{t},\label{eq:Duffing}
\end{equation}
where $A,\omega\in\mathbb{R}$ and $\sigma\neq0$ are
parameters. In the Benzi-Parisi-Sutera-Vulpiani climate change stochastic resonance model, 
$\omega =2\pi /10^5$ and the two stable equilibrium climates are distanced by $10K$. 
The stochastic differential equation (\ref{eq:Duffing}) is a normalised equation of the 
Benzi-Parisi-Sutera-Vulpiani model. According to Corollary \ref{zhao2019.4a}, there exists a 
unique periodic measure which is geometric ergodic.
\end{example}
\begin{rem}
Through the theory of non-autonomous RDS, \cite{CherubiniStochRes}
gave the existence and uniqueness of the periodic measure for (\ref{eq:Duffing})
in one dimension. Note that Theorem \ref{thm:OddPolyDrift} goes further
than \cite{CherubiniStochRes} to infer that the convergence is actually
geometric. Moreover, Theorem \ref{thm:OddPolyDrift} gives the other
types of converges presented in Theorem \ref{thm:Geometric_PM_on_E}.
To our knowledge, this paper contains the first proof of the geometric
ergodicity of the stochastic overdamped Duffing Oscillator. We note
also that the approach we have taken works in multidimensional case
and is completely different to the that of \cite{CherubiniStochRes}.
As mentioned in the introduction, we expect our approach can be extended
to the infinite dimensional setting of SPDEs.
\end{rem}

\subsubsection{Gradient Systems}

In this section, we give results for the existence and uniqueness
of geometric periodic measures for stochastic $T$-periodic gradient
systems. These are SDEs of the form 

\begin{equation}
dX_{t}=-\nabla V(t,X_{t})dt+\sigma(X_{t})dW_{t},\label{eq:GradientSDE}
\end{equation}
where $V\in C^{1,2}(\mathbb{R}^{+}\times\mathbb{R}^{d})$ is $T$-periodic,
$\nabla=(\partial_{1},\cdot\cdot\cdot,\partial_{d})$ is the spatial
gradient operator, $W_{t}$ denotes a $d$-dimensional Brownian motion
and $\sigma:\mathbb{R}^{d}\rightarrow\mathbb{R}^{d\times d}$. Note
that the $T$-periodicity of $V$ implies the $T$-periodicity of
$\nabla V$ is $T$-periodic hence the gradient SDE (\ref{eq:GradientSDE})
is $T$-periodic. 

Gradient systems arise naturally in physical applications, where $V$
is referred to as the potential function \cite{Gardiner_StochMethods,MattinglyStuartHigham_ErgodicityForSDEs,Pavliotis_LangevinTextbook}.
Indeed examples of $T$-periodic gradient systems, includes the periodic
forced Ornstein-Uhlenbeck from Example \ref{exa:PeriodicOUExample}
derived from $V(t,x)=\frac{\alpha}{2}\left(x-\frac{S(t)}{\alpha}\right)^{2}$
and the Duffing Oscillator from Example \ref{Exa:DuffingOsc} derived
from double-well potential $V(t,x)=\frac{1}{4}x^{4}-\frac{1}{2}x^{2}+A\cos(\omega t)x$.
In fact, it is easy to verify that the Theorem \ref{thm:OddPolyDrift}
is a special case of gradient SDEs derived from potential $V(t,x)=\sum_{i=1}^{d}\sum_{k=1}^{2p_{i}}\frac{S_{k}^{i}(t)}{k+1}x_{i}^{k+1}$.
While these examples are weakly dissipative where the Euclidean norm
is a suitable Foster-Lyapunov function satisfying (\ref{eq:GeometricDriftCdn}),
in general, finding a Foster-Lyapunov function satisfying (\ref{eq:GeometricDriftCdn})
for a given SDE is generally non-trivial (if at all possible) particularly
in higher dimensions. A mathematical advantage of gradient systems
is that $V$ itself is a natural choice of Foster-Lyapunov function
to satisfy (\ref{eq:GeometricDriftCdn}). This is apparent by observing
the generator of (\ref{eq:GradientSDE}) is given by 
\begin{equation}
\mathcal{L}(t)V(t,x)=\partial_{t}V(t,x)-\left\lVert \nabla V(t,x)\right\rVert ^{2}+\frac{1}{2}\sum_{i,j=1}^{d}\left(\sigma\sigma^{T}(x)\right)_{ij}\partial_{ij}^{2}V(t,x),\label{eq:GradientGenerator}
\end{equation}
and exploiting the norm term. 

For autonomous gradient SDEs derived from a norm-like potential $V(t,x)=V(x)$
and noise proportional to the identity $\sigma\in\mathbb{R}^{+}\backslash\{0\}$,
it is well-known \cite{MarkowichVillaniTrendToEquil,Gardiner_StochMethods,Pavliotis_LangevinTextbook}
that the invariant measure has a particularly simple form and is given
by (upon normalisation) $\pi(\Gamma)=\int_{\Gamma}\exp\left(-\frac{2V(x)}{\sigma^{2}}\right)dx$
for $\Gamma\in\mathcal{B}$. However, due to the intricate interplay
between stochasticity and periodicity, periodic measures (with a minimal
positive period) does not have such simple expression. Indeed the
periodic measure (\ref{eq:PM_OU}) from Example \ref{exa:PeriodicOUExample}
does not take this simple form i.e. $\rho_{s}(\Gamma)\neq\int_{\Gamma}\exp\left(-\frac{2V(s,x)}{\sigma^{2}}\right)dx$. 

The following corollary of Theorem \ref{thm:GeometricConvergence-PeriodicMeasure}
is generally simple to verify to yield gradient SDEs with a geometric
periodic measure.
\begin{cor}
\label{cor:GradientSDEViaV}Assume $\sigma$ satisfy (\ref{eq:LinearGrowth_sigma}),
(\ref{eq:Sigma_and_inverse_bounded}) and (\ref{eq:LocallySmoothSigma}).
Let $V\in C^{1,2}(\mathbb{R}^{+}\times\mathbb{R}^{d})$ be a norm-like
function such that for all $n\in\mathbb{N}$
\[
\partial^{\alpha}V\quad\text{bounded on }\mathbb{R}^{+}\times B_{n},\alpha\in\mathbb{N}^{d+1},\lvert\alpha\rvert\in\{1,d+1\},
\]
and (\ref{eq:GeometricDriftCdn}) holds, where $\mathcal{L}(t)$ is
given by (\ref{eq:GradientGenerator}). Then the results of Theorem
\ref{thm:Geometric_PM_on_E} holds for SDE (\ref{eq:GradientSDE}).
\end{cor}
While Corollary \ref{cor:GradientSDEViaV} covers all the examples
considered thus far, it applies to a wider class of SDEs than that
of weakly dissipative systems. In the next proposition, we use Corollary
\ref{cor:GradientSDEViaV} to extend the case of Theorem \ref{thm:OddPolyDrift}
when $p_{i}=\text{const}$ for all $i$ and allowing for products
of the spatial variables. It does not aim to be most general however
suffices a range of applications. We shall employ more standard multi-index
notation: for spatial variables $x=(x_{1},\cdot\cdot\cdot,x_{d})$
and multi-index $\alpha\in\mathbb{N}^{d}$ , define $x^{\alpha}:=x_{1}^{\alpha_{1}}\cdot\cdot\cdot x_{d}^{\alpha_{d}}$.
For $\alpha,\beta\in\mathbb{N}^{d}$, we have the partial ordering
$\alpha\geq\beta$ if $\alpha_{i}\geq\beta_{i}$ for each $1\leq i\leq d$.
We define the standard tuple basis $e_{i}=(0,\cdot\cdot\cdot,1,\cdot\cdot\cdot,0)$
where the $1$ appears on the $i$'th index. For fixed $\beta\in\mathbb{N}^{d}$,
we define $\sum_{\alpha\geq\beta}^{N}C_{\alpha}:=\sum_{\alpha\geq\beta}^{|\alpha|\leq N}C_{\alpha}.$
Recall standard asymptotic notation where for functions $f_{1},f_{2},g:\mathbb{R}^{d}\rightarrow\mathbb{R}$,
we write $\max\{f_{1},f_{2}\}=o(g)$ if $\lim_{\lVert x\rVert\rightarrow\infty}\frac{\max\{\lvert f_{1}(x)\rvert,\lvert f_{2}(x)\rvert\}}{g(x)}=0.$
This implies that for any $\epsilon>0$, there exists $R>0$ such
that 
\begin{equation}
\max\{\lvert f_{1}(x)\rvert,\lvert f_{2}(x)\rvert\}\leq\epsilon\lvert g(x)\lvert,\quad x\in B_{R}^{c}.\label{eq:Little_o_bound}
\end{equation}

\begin{prop}
\label{Prop:MultidimensionalPotential}Assume $\sigma$ satisfy (\ref{eq:LinearGrowth_sigma}),
(\ref{eq:Sigma_and_inverse_bounded}) and (\ref{eq:LocallySmoothSigma}).
Let $\left\{ S_{\alpha}(t)\right\} _{\alpha\in\mathbb{N}^{d}}$ be
continuously differentiable $T$-periodic functions and $\{S_{i}\}_{i=1}^{d}$
are strictly positive constants. Then the gradient system (\ref{eq:GradientSDE})
with potential
\[
V(t,x)=\sum_{i=1}^{d}S_{i}x_{i}^{p}+\sum_{|\alpha|=0}^{p-1}S_{\alpha}(t)x^{\alpha},\quad p\in2\mathbb{N}:=\{2,4,...,\},
\]
satisfies Corollary \ref{cor:GradientSDEViaV} hence the results of
Theorem \ref{thm:Geometric_PM_on_E} holds.
\end{prop}
\begin{proof}
We compute
\[
\begin{cases}
\partial_{t}V=\sum_{|\alpha|=0}^{p-1}\dot{S}_{\alpha}x^{\alpha},\\
\partial_{i}V=S_{i}px_{i}^{p-1}+\sum_{\alpha\geq e_{i}}^{p-1}\alpha_{i}S_{\alpha}x^{\alpha-e_{i}},\\
\partial_{ii}^{2}V=p(p-1)S_{i}x_{i}^{2p-2}+\sum_{\alpha\geq2e_{i}}^{p-1}\alpha_{i}(\alpha_{i}-1)S_{\alpha}x^{(\alpha-2e_{i})},\\
\partial_{ij}^{2}V=\sum_{\alpha\geq e_{i}+e_{j}}^{p-1}S_{\alpha}\alpha_{i}\alpha_{j}x^{\alpha-e_{i}-e_{j}}, & i\neq j.
\end{cases}
\]
So
\[
\lVert\nabla V\rVert^{2}=\sum_{i=1}^{d}(\partial_{i}V)^{2}=\sum_{i=1}^{d}\left[S_{i}^{2}p^{2}x_{i}^{2p-2}+2S_{i}p\sum_{\alpha\geq e_{i}}^{p-1}\alpha_{i}S_{\alpha}x^{\alpha+(p-2)e_{i}}+\left(\sum_{\alpha\geq e_{i}}^{p-1}\alpha_{i}S_{\alpha}x^{\alpha-e_{i}}\right)^{2}\right].
\]
Note that $V$,$\partial_{t}V$, $\partial_{ij}^{2}V$ and $\left(\lVert\nabla V\rVert^{2}-\sum_{i=1}^{d}S_{i}^{2}p^{2}x_{i}^{2p-2}\right)$
has maximum order $p$, $p-1$, $p-3$ and $2p-3$ respectively. Our
assumptions ensures that $\max_{\alpha\in\mathbb{N}^{d}}(\sup_{t\in\mathbb{R}}\lvert S_{\alpha}(t)\rvert)<\infty$
and $\max_{i,j}\sup_{x\in\mathbb{R}^{d}}a_{ij}(x)<\infty$. Since
higher even powers dominates lower powers i.e. $x^{\alpha}=o(\sum_{i=1}^{d}c_{i}x_{i}^{2n})$
where $c_{i}>0$ and $\lvert\alpha\rvert<2n$ where $n\in\mathbb{N}$,
we have for any $\lambda>0$
\[
\max\left\{ \lambda V,\partial_{t}V,\partial_{ij}^{2}V,\left(\lVert\nabla V\rVert^{2}-\sum_{i=1}^{d}S_{i}^{2}p^{2}x_{i}^{2p-2}\right)\right\} =o\left(\sum_{i=1}^{d}S_{i}^{2}p^{2}x_{i}^{2p-2}\right),\quad2<p\in2\mathbb{N}.
\]
Then for $2<p\in2\mathbb{N}$, by (\ref{eq:Little_o_bound}), for
any $\epsilon\in(0,\frac{1}{4})$, there exists $R>0$ such that 
\begin{align*}
\mathcal{L}(t)V+\lambda V & \leq\lvert\partial_{t}V\rvert-\lVert\nabla V\rVert^{2}+\frac{1}{2}\lvert\sum_{i,j=1}^{d}a_{ij}V\rvert+\lambda V\leq(4\epsilon-1)\left(\sum_{i=1}^{d}S_{i}^{2}p^{2}x_{i}^{2p-2}\right)\leq0,\quad x\in B_{R}^{c}.
\end{align*}
By continuity, $\mathcal{L}(t)V+\lambda V$ is bounded on $B_{R}$.
Hence (\ref{eq:GeometricDriftCdn}) is satisfied. For $p=2$ where
$V$ and $\sum_{i=1}^{d}S_{i}^{2}p^{2}x_{i}^{2p-2}$ are of the same
order, the same calculations holds provided one restricts $0<\lambda<4\min_{i}S_{i}^{2}.$ 
\end{proof}
In physics literature, ``periodically forced'' or ``periodically
driven'' generally refers to the addition of a periodic term on the
drift which otherwise be autonomous i.e. $b(t,x)=b_{0}(x)+S(t)$ for
some periodic function $S$ and drift $b_{0}$ independent of $t$.
Particular instances of Proposition \ref{Prop:MultidimensionalPotential}
include periodically-forced systems such Example \ref{exa:PeriodicOUExample}
and Example \ref{Exa:DuffingOsc}. Mentioned examples so far are systems
with polynomial potentials. While polynomial approximation of potentials
(by Weierstrass approximation theorem for instance) can be effective
for practical reasons, we consider periodically forced gradient systems
that need not be derived from a polynomial potential. We remark that
periodically forced gradient SDEs occurs in physical applications
and phenomena as we have already seen in previous examples. For further
discussions on periodically forced stochastic systems, we refer readers
to \cite{JungPerioidcallyDrivenSys} for theory and applications.

Consider the following (autonomous) gradient SDE on $\mathbb{R}^{d}$
\begin{equation}
dX_{t}=-\nabla U(X_{t})+\sigma(X_{t})dW_{t},\label{eq:AutonomousGradSystem}
\end{equation}
where $\sigma$ satisfy (\ref{eq:LinearGrowth_sigma}), (\ref{eq:Sigma_and_inverse_bounded})
and (\ref{eq:LocallySmoothSigma}) and $U\in C^{2}(\mathbb{R}^{d},\mathbb{R}^{+})$
satisfies the (autonomous) geometric drift condition 
\begin{equation}
LU\leq C-\lambda U\quad\text{on }\mathbb{R}^{d},\label{eq:AutonGeometricDrift}
\end{equation}
where $C\geq0,\lambda>0$ are constants and $\mathcal{L}$ is the
infinitesimal generator of (\ref{eq:AutonomousGradSystem}) given
by
\[
Lf(x)=-\langle\nabla U(x),\nabla f(x)\rangle+\frac{1}{2}\sum_{i,j=1}^{d}(\sigma\sigma^{T}(x))_{ij}\partial_{ij}^{2}f(x),\quad f\in C^{2}(\mathbb{R}^{d}).
\]
This classical geometric drift condition yields the existence, uniqueness
and ergodicity of an invariant measure. The context of the next lemma
sufficiently yields a geometric periodic measure when the autonomous
gradient system is periodically forced. Essentially, the autonomous
system retains its stability up to replacing its invariant measure
for a periodic measure with a minimal positive period. Note that we
do not impose any particular form imposed on the potential, hence
more general than polynomials. We note that the assumptions are easily
satisfied for many practical systems.
\begin{prop}
Let $U\in C^{2}(\mathbb{R}^{d},\mathbb{R}^{+})$ be a norm-like potential
satisfying (\ref{eq:AutonGeometricDrift}) and that for any $c_{1},c_{2}>0$,
there exists a compact set $K=K(c_{1},c_{2})\subset\mathbb{R}^{d}$
such that 
\[
c_{1}\lVert x\rVert\leq c_{2}U(x)\quad x\in K^{c}.
\]
Then for any $T$-periodic ($T>0$) continuously differentiable function
$S:\mathbb{R}^{+}\rightarrow\mathbb{R}^{d}$, the periodically forced
gradient SDE
\[
dX_{t}=-\left[\nabla U(X_{t})+S(t)\right]dt+\sigma(X_{t})dW_{t}
\]
possesses a unique geometric periodic measure with a minimal positive
period.
\end{prop}
\begin{proof}
By Theorem \ref{thm:GeometricConvergence-PeriodicMeasure}, we verify
$V(t,x)=U(x)-\langle S(t),x\rangle$ satisfies (\ref{eq:GeometricDriftCdn}).
By the assumptions on $U$ and $S$, it is clear that $V\in C^{1,2}(\mathbb{R}^{+}\times\mathbb{R}^{d})$
is a $T$-periodic norm-like potential satisfying the regularity assumptions
of Corollary \ref{cor:GradientSDEViaV}. Since $\partial_{ij}^{2}V=\partial_{ij}^{2}U$,
we compute that
\begin{align*}
\mathcal{L}(t)V & =-\langle\dot{S},x\rangle-\langle\nabla U(X_{t})+S(t),\nabla U(X_{t})-S(t)\rangle+\sum_{i,j=1}^{d}(\sigma\sigma^{T}(x))_{ij}\partial_{ij}^{2}V\\
 & =-\langle\dot{S},x\rangle-\lVert\nabla U\rVert^{2}+\lVert S\rVert^{2}+\sum_{i,j=1}^{d}(\sigma\sigma^{T}(x))_{ij}\partial_{ij}^{2}U\\
 & =\lVert S\rVert^{2}-\langle\dot{S},x\rangle+LU.
\end{align*}
As $U$ satisfies the geometric drift condition, by picking any fixed
$\lambda^{-}\in(0,\lambda),$ we have 
\begin{align*}
\mathcal{L}(t)V & \leq\lVert S\rVert^{2}-\langle\dot{S},x\rangle+C-\lambda U\\
 & =\lVert S\rVert^{2}-\langle\dot{S},x\rangle+C-(\lambda-\lambda^{-})U-\lambda^{-}U+(\lambda-\lambda^{-})\langle S,x\rangle-(\lambda-\lambda^{-})\langle S,x\rangle\\
 & =\lVert S\rVert^{2}-\langle\dot{S}+(\lambda-\lambda^{-})S,x\rangle+C-(\lambda-\lambda^{-})V-\lambda^{-}U\\
 & \leq\lVert S\rVert^{2}+\lVert\dot{S}+(\lambda-\lambda^{-})S\rVert_{\infty}\lVert x\rVert+C-(\lambda-\lambda^{-})V-\lambda^{-}U,
\end{align*}
where $\lVert\dot{S}+(\lambda-\lambda^{-})S\rVert_{\infty}:=\sup_{s\in[0,T]}\lVert\dot{S}(s)+(\lambda-\lambda^{-})S(s)\rVert<\infty$
as $S$ and $\dot{S}$ are bounded. Then, by assumption with $c_{1}=\lVert\dot{S}+(\lambda-\lambda^{-})S\rVert_{\infty}$
and $c_{2}=\lambda^{-}$, we have a compact set $K\subset\mathbb{R}$
such that 

\[
c:=\sup_{x\in K}\left(\lVert\dot{S}+(\lambda-\lambda^{-})S\rVert_{\infty}\lVert x\rVert-\lambda^{-}U\right)<\infty.
\]
Hence $\mathcal{L}(t)V\leq\left(C+c+\lVert S\rVert^{2}\right)-\left(\lambda-\lambda^{-}\right)V$
i.e. the geometric drift condition (\ref{eq:GeometricDriftCdn}) is
satisfied. Since $S$ has a minimal positive period $T>0$, by Proposition
\ref{prop:CoExistence_with_IM} the periodic measure will have a minimal
positive period. 
\end{proof}

\subsection{Langevin Dynamics}

Langevin equations originated to model noisy molecular systems and
many other physical phenomena. As such, we expect applications to
the physical sciences. In fact, we shall see it extends easily from
stochastic gradient systems in an ``overdamped'' limit and applies
immediately to the stochastic periodically-forced harmonic oscillator.
We refer the reader to \cite{Zwanzig,Pavliotis_LangevinTextbook}
for further applications, details and derivations of Langevin equations.
Akin to earlier sections, we give sufficient conditions for the existence,
uniqueness and geometric convergence of a periodic measure for $T$-periodic
Langevin equations. We study Langevin equations of the form 

\begin{equation}
md\dot{q}_{t}=\left(F(t,q_{t})-\gamma\dot{q}_{t}\right)dt+\sigma dW_{t},\label{eq:StochasticN2L}
\end{equation}
with position $q_{t}\in\mathbb{R}^{d}$ , velocity $\dot{q}_{t}\in\mathbb{R}^{d}$,
acceleration $\ddot{q}_{t}\in\mathbb{R}^{d}$, constant mass $m>0$
, time-dependent force $F:\mathbb{R}^{+}\times\mathbb{R}^{d}\rightarrow\mathbb{R}^{d}$,
$d$-dimensional Brownian motion $W_{t}$ and constant matrix $\sigma\in GL(\mathbb{R}^{d})$.
For $\gamma\geq0$, $\gamma\dot{q}_{t}$ is understood as the frictional
force of the system. The proportional constant $\gamma$ is referred
as the damping constant. Without loss of generality, we take mass
to be unit i.e. $m=1$. 

Denote momentum $p_{t}=\dot{q}_{t}$, then (\ref{eq:StochasticN2L})
can be rewritten as a system of first order SDEs

\begin{equation}
\begin{cases}
dq_{t}=p_{t}dt,\\
dp_{t}=\left(-\gamma p_{t}+F(t,q_{t})\right)dt+\sigma dW_{t}.
\end{cases}\label{eq:LangevinGeneral}
\end{equation}
In phase space coordinates $X_{t}=(q_{t},p_{t})\in\mathbb{R}^{2d}$
this can be rewritten as

\begin{equation}
dX_{t}=b(t,X_{t})dt+\Sigma d\mathcal{W}_{t},\label{eq:LangevinEquationGeneral_X}
\end{equation}
where 
\begin{equation}
b(t,x)=b(t,q,p)=\left(\begin{array}{c}
p\\
-\gamma p+F(t,q)
\end{array}\right)\in\mathbb{R}^{2d},\quad\Sigma=\left(\begin{array}{cc}
0 & 0\\
0 & \sigma
\end{array}\right)\in\mathbb{R}^{2d\times2d},\quad\mathcal{W}_{t}=\left(\begin{array}{c}
0\\
W_{t}
\end{array}\right).\label{eq:DriftPhaseSpace}
\end{equation}
On a physical level, observe that the noise is degenerate in that
the noise affects $q_{t}$ only through $p_{t}$. Formally, Langevin
SDE (\ref{eq:LangevinEquationGeneral_X}) is degenerate since $\Sigma\notin GL(\mathbb{R}^{2d})$.
Resultantly, Theorem \ref{thm:DPZIrred} does not apply. Hence in
this current paper, we only study Langevin dynamics with only additive
noise. It will be of future works to study the situation with multiplicative
noise.

Written in phase space coordinates, it is clear that (\ref{eq:StochasticN2L})
has unique solution provided $b$ and $\sigma$ are Lipschitz. Labelling
$x=(q,p)=(x_{1},..,x_{2d})$, the infinitesimal generator is given
by 
\begin{align}
\mathcal{L}(t)f(t,x) & =\partial_{t}f+\langle p,\nabla_{q}f\rangle+\langle-\gamma p+F,\nabla_{p}f\rangle+\frac{1}{2}\sum_{i,j=1}^{d}(\sigma\sigma^{T})_{ij}\partial_{p_{i}p_{j}}^{2}f,\quad f\in C^{2,1}(\mathbb{R}^{+}\times\mathbb{R}^{2d}),\label{eq:LangevinGenerator}
\end{align}
where $\nabla_{q}:=(\partial_{q_{1}},\cdot\cdot\cdot,\partial_{q_{d}})^{T}$
and similarly $\nabla_{p}:=(\partial_{p_{1}},\cdot\cdot\cdot,\partial_{p_{d}})^{T}$.
\begin{rem}
We remark that in physical applications concerning small particles,
the mass is typically small. This suggest the inertia term $m\ddot{q}_{t}$
can be neglected. Hence, not rigorously, the dynamics (\ref{eq:StochasticN2L})
can be well-approximated by
\[
0=F(t,q_{t})-\gamma\dot{q}_{t}+\sigma dW_{t}.
\]
i.e. reduced to SDEs studied earlier in this section. Suggesting that
Langevin equations may be studied with multiplicative noise in the
context of small particles. A particular source of interesting dynamics
and applications is the case when $F(t,q)=-\nabla_{q}V(t,q)$ for
some potential $V(t,q)$ and so the Langevin equations are gradient
systems (provided $\gamma>0$). Such systems without inertia are called
overdamped Langevin dynamics.
\end{rem}
With the inapplicability of Theorem \ref{thm:DPZIrred}, we the following
irreducibility lemma for non-autonomous Langevin equation. The can
be done by a similar method as in \cite{MattinglyStuartHigham_ErgodicityForSDEs},
so it is omitted here.
\begin{lem}
\label{lem:LangevinReachability}Consider $T$-periodic Langevin equation
(\ref{eq:LangevinGeneral}) with locally Lipschitz $F$. Assume there
exists a norm-like function $V$ satisfying (\ref{eq:RegularityCondition}).
Then the Markov transition kernel satisfies $P(s,t,x,\Gamma)>0$ for
any $s<t<\infty,x\in\mathbb{R}^{2d}$ and non-empty open $\Gamma\in\mathcal{B}(\mathbb{R}^{2d})$.
\end{lem}
Thus we have the following Langevin counterpart of Theorem \ref{thm:GeometricConvergence-PeriodicMeasure}.
\begin{thm}
Consider $T$-periodic Langevin equation (\ref{eq:LangevinEquationGeneral_X})
with $F$ satisfying (\ref{eq:LocallySmoothDrift}) (in place of $b$).
Assume there exists a norm-like function $V\in C^{1,2}(\mathbb{R}^{+}\times\mathbb{R}^{2d},\mathbb{R}^{+})$
satisfying (\ref{eq:GeometricDriftCdn}) where $\mathcal{L}$ is given
by (\ref{eq:LangevinGenerator}). Then there exists a unique geometric
periodic measure $\rho:\mathbb{R}^{+}\rightarrow\mathcal{P}(\mathbb{R}^{2d})$
satisfying the convergences from Theorem \ref{thm:Geometric_PM_on_E}.
\end{thm}
\begin{proof}
Let $\sigma_{i}$ denote the $i$'th column of $\sigma$ then $\Sigma_{i}=(0,\sigma_{i})^{T}$
denote the $i$'th column of $\Sigma$. Denoting $Id\in\mathbb{R}^{d\times d}$
to be the identity matrix, observe that Lie bracket 
\[
[\Sigma_{i},b]=(Db)\Sigma_{i}=\left(\begin{array}{cc}
0 & Id\\
-d^{2}F(t,q) & -\gamma Id
\end{array}\right)\left(\begin{array}{c}
0\\
\sigma_{i}
\end{array}\right)=\left(\begin{array}{c}
\sigma_{i}\\
-\gamma\sigma_{i}
\end{array}\right).
\]
Since $\sigma\in GL$($\mathbb{R}^{d}$), the columns $\sigma_{i}$
are linear independent. Hence the Lie algebra generated by $\Sigma_{i}$
and $b$ spans $\mathbb{R}^{2d}$. By the assumptions on $F$, $b$
satisfies (\ref{eq:LocallySmoothDrift}). Hence together with Foster-Lyapunov
function $V$, there exists a smooth density $p(s,t,x,y)$ with respect
to $\Lambda$by Theorem 1 of \cite{HopfnerLocherbachThieullenTimeDepLocalHormanders}
is satisfied. $F$ is locally Lipschitz hence with Lemma \ref{lem:LangevinReachability},
Theorem \ref{thm:DensityContinuityforLCD} holds. Hence the assumptions
of Theorem \ref{thm:Geometric_PM_on_E} are satisfied.
\end{proof}

\section{Density of Periodic Measures }

Similar to invariant measures, it is interesting and important to
know when periodic measures possess a density with respect to Lebesgue
measure. In this section, we show that the density of the periodic
measure for $T$-periodic SDEs on $\mathbb{R}^{d}$ necessarily and
sufficiently satisfies a Fokker-Planck PDE of an ``initial-terminal''
kind. For the length of this paper, this paper does not include the
existence of density of periodic measures in the general case, while
we do provides an explicit example for the periodically forced Ornstein-Uhlenbeck
process. We will study this in future publications. 

In previous sections, we have predominantly been focused initial state,
here we change our perspective to the forward spatial variable. As
such, at the risk of confusion, we interchange the roles of $x$ and
$y$ i.e. we take $y\in\mathbb{R}^{d}$ to be the initial state and
$x\in\mathbb{R}^{d}$ to be the forward time variable. Let $P(s,t,y,\cdot)$
be a two-parameter Markov transition kernel with a density $p(s,t,y,x)$
and $\mu_{s},\mu_{t}\in\mathcal{P}(\mathbb{R}^{d})$ for $s<t$ with
respective density $q(s,\cdot),q(t,\cdot)\in L^{1}(\mathbb{R}^{d})$
satisfying $\mu_{t}=P^{*}(s,t)\mu_{s}$. Then, by Fubini's theorem,
$q$ satisfies

\begin{equation}
q(t,x)=\int_{\mathbb{R}^{d}}p(s,t,y,x)q(s,y)dy.\label{eq:PushForwardDensity}
\end{equation}
It is well-known that $q$ satisfies the following Fokker-Planck equation
\[
\begin{cases}
\partial_{t}q=L^{*}(t)q,\\
\lim_{t\downarrow s}q(t,\cdot)=q(s,\cdot).
\end{cases}
\]
where $L^{*}(t)$ is the Fokker-Planck operator given by 
\begin{equation}
L^{*}(t)q=-\sum_{i=1}^{d}\partial_{x_{i}}(b_{i}(t,x)q)+\frac{1}{2}\sum_{i,j=1}^{d}\partial_{x_{i}x_{j}}^{2}\left(\left(\sigma\sigma^{T}(t,x\right)_{ij}q\right).\label{eq:FokkerPlanckEqn}
\end{equation}
In this section, we will always assume that the operator $L^{*}(t)$
is uniformly elliptic i.e. there exists $\lambda>0$ such that $\langle\xi,\sigma\sigma^{T}(t,x)\xi\rangle\geq\lambda\lVert\xi\rVert^{2}$
for all $(t,x)\in\mathbb{R}^{+}\times\mathbb{R}^{d}$ and $\xi\in\mathbb{R}^{d}$.

In the following, we shall use the notation $X\sim q$ to mean the
random variable $X$ is distributed by probability density $q\in L^{1}(\mathbb{R}^{d})$.
For random variables $X^{0}$ and $X^{1}$, we write $X^{0}\sim X^{1}$
if they have the same distribution. We state and prove the following
useful lemma. 
\begin{lem}
\label{lem:DistributionModT} Assume $\left(X_{t}^{0}\right)_{t\geq s},\left(X_{t}^{1}\right)_{t\geq s+T}$
are two processes satisfying the $T$-periodic SDE (\ref{eq:NonAutonSDE}).
If $X_{s}^{0}\sim X_{s+T}^{1}$ then $X_{s+t}^{0}\sim X_{s+T+t}^{1}$
for all $t\geq0.$
\end{lem}
\begin{proof}
For concreteness, let $X_{s}^{0}\sim X_{s+T}^{1}\sim q\in L^{1}(\mathbb{R}^{d})$
and $p^{0}(s+t,\cdot)$ denote the distribution of $X_{s+t}^{0}$
and similarly $p^{1}(s+T+t,\cdot)$ for $X_{s+T+t}^{1}$. Then $p^{k}$
satisfies the Fokker-Planck equation i.e. for $k=0,1$ and $t\geq0$
\[
\begin{cases}
\partial_{t}p^{k}(t+kT,x)=L^{*}(t+kT)p^{k}(t+kT,x),\\
p^{k}(s+kT,\cdot)=q.
\end{cases}
\]
It is clear that $L^{*}(t)=L^{*}(t+T)$ by the $T$-periodic coefficients.
By the linearity of the Fokker-Planck operator, it is easy to see
that $\hat{p}(t,\cdot):=p^{0}(s+t,\cdot)-p^{1}(s+t+T,\cdot)$ satisfies

\[
\begin{cases}
\partial_{t}\hat{p}=L^{*}(t)\hat{p} & t\geq0,\\
\hat{p}(0,\cdot)=0.
\end{cases}
\]
Then an application of parabolic maximum principle or otherwise yields
that $\hat{p}(t,\cdot)=0$ for all $t\geq0$ is the only physical
solution. Hence concluding $p^{0}(s+t,\cdot)=p^{1}(s+T+t,\cdot)$
for all $t\geq0$.
\end{proof}
With Lemma \ref{lem:DistributionModT}, we are now ready to state
the main result of this section.
\begin{thm}
\label{thm:FokkerPlanckPeriodicSoln} Consider $T$-periodic SDE (\ref{eq:NonAutonSDE})
with continuous coefficients. For $q\in C^{1,2}(\mathbb{R}^{+}\times\mathbb{R}^{d})\cap L^{1}(\mathbb{R}^{d})$
define $\rho:\mathbb{R}^{+}\rightarrow\mathcal{P}(\mathbb{R}^{d})$
by 

\[
\rho_{t}(\Gamma)=\frac{1}{\left\lVert q(t,\cdot)\right\rVert _{L^{1}(\mathbb{R}^{d})}}\int_{\Gamma}q(t,x)dx,\quad t\geq0.
\]
Then $\rho$ is a $T$-periodic measure if and only if 
\begin{equation}
\partial_{t}q=L^{*}(t)q,\quad q(0,\cdot)=q(T,\cdot).\label{eq:PeriodicFokkerPlanck}
\end{equation}
Hence, if (\ref{eq:PeriodicFokkerPlanck}) has a unique solution then
there is a unique periodic measure with density $q$. 
\end{thm}
\begin{proof}
For notational convenience and without loss of generality, we let
$q(t,x)$ be normalised. Assume $\rho$ is a $T$-periodic measure,
then by definition, $\rho_{t}=\rho_{t+T}$ for all $t\geq0$ i.e.
\[
\int_{\Gamma}q(t,x)dx=\rho_{t}(\Gamma)=\rho_{t+T}(\Gamma)=\int_{\Gamma}q(t+T,x)dx,\quad\Gamma\in\mathcal{B}(\mathbb{R}^{d}).
\]
As this holds for any $\Gamma\in\mathcal{B}(\mathbb{R}^{d})$, it
follows that $q(t,\cdot)=q(t+T,\cdot)$. On the other hand, it is
well know that $p(s,t,y,x)$ satisfies the Fokker-Planck equation
\[
\partial_{t}p(s,t,y,x)=L^{*}(t)p(s,t,y,x),
\]
 We take derivative with respect to $t$ on both sides of (\ref{eq:PushForwardDensity}),
we have 
\begin{eqnarray*}
\partial_{t}q(t,x) & = & \int_{\mathbb{R}^{d}}\partial_{t}p(s,t,y,x)q(s,y)dy\\
 & = & \int_{\mathbb{R}^{d}}L^{*}(t)p(s,t,y,x)q(s,y)dy\\
 & = & \int_{\mathbb{R}^{d}}-\sum_{i=1}^{d}\partial_{x_{i}}(b_{i}(t,x)p(s,t,y,x))q(s,y)dy\\
 &  & +\int_{\mathbb{R}^{d}}\frac{1}{2}\sum_{i,j=1}^{d}\partial_{x_{i}x_{j}}^{2}\left(\left(\sigma\sigma^{T}(t,x\right)_{ij}p(s,t,y,x)\right)q(s,y)dy\\
 & := & A+B.
\end{eqnarray*}
For the first term, we have 
\begin{eqnarray*}
A & = & -\sum_{i=1}^{d}\int_{\mathbb{R}^{d}}\left[\partial_{x_{i}}(b_{i}(t,x))p(s,t,y,x)+b_{i}(t,x)\partial_{x_{i}}(p(s,t,y,x))\right]q(s,y)dy\\
 & = & -\sum_{i=1}^{d}\partial_{x_{i}}(b_{i}(t,x))\int_{\mathbb{R}^{d}}p(s,t,y,x)q(s,y)dy-\sum_{i=1}^{d}b_{i}(t,x)\partial_{x_{i}}\int_{\mathbb{R}^{d}}p(s,t,y,x)q(s,y)dy\\
 & = & -\sum_{i=1}^{d}\partial_{x_{i}}(b_{i}(t,x))q(t,x)-\sum_{i=1}^{d}b_{i}(t,x)\partial_{x_{i}}q(t,x)\\
 & = & -\sum_{i=1}^{d}\partial_{x_{i}}(b_{i}(t,x)q(t,x)).
\end{eqnarray*}
Similarly, for the second term, we have 

\[
B=\frac{1}{2}\sum_{i,j=1}^{d}\partial_{x_{i}x_{j}}^{2}\left(\left(\sigma\sigma^{T}(t,x\right)_{ij}q(t,x)\right).
\]
Therefore, the density function $q(\cdot,\cdot)$ satisfies 
\[
q(t,\cdot)=q(t+T,\cdot),\quad\partial_{t}q=L^{*}(t)q,\quad\text{for all }s\geq0.
\]
By Lemma \ref{lem:DistributionModT}, it suffices that this PDE holds
specifically for $t=0$ hence we have (\ref{eq:PeriodicFokkerPlanck}). 

To prove the converse, we first note that Lemma \ref{lem:DistributionModT}
yields that $q(t,\cdot)=q(t+T,\cdot)$ for all $t\geq0$. Thus $\rho$
is $T$-periodic. By (\ref{eq:PushForwardDensity}) and Fubini's theorem,
it is clear that 
\begin{align*}
P^{*}(s,t)\rho_{s}(\Gamma) & =\int_{\mathbb{R}^{d}}\left[\int_{\Gamma}p(s,t,y,x)dx\right]q(s,y)dy=\int_{\Gamma}q(t,x)dx=\rho_{t}(\Gamma),\quad\Gamma\in\mathcal{B}(\mathbb{R}^{d}),
\end{align*}
concluding that $\rho$ is a $T$-periodic measure. 
\end{proof}
There is an ``alternative'' way to arrive the PDE of Theorem \ref{thm:FokkerPlanckPeriodicSoln}
as seen in \cite{JungNumericalSchemeForLiftedFokkerPlanck}. By considering
lifted coordinates $(t,X_{t})$, one can consider stationary solutions
of the lifted Fokker-Planck operator $\mathcal{L}^{*}$ i.e. $q(t,x)$
satisfying
\begin{align}
\mathcal{L}^{*}(t)q & :=-\partial_{t}q(t,x)+L^{*}(t)q(t,x)=0.\label{eq:StationaryLiftedFP}
\end{align}
This is equivalent to (\ref{eq:PeriodicFokkerPlanck}) upon rearranging.
However, this approach does not naturally imposes any boundary conditions,
hence is not sufficient for $q$ to be the density of the periodic
measure. Theorem \ref{thm:FokkerPlanckPeriodicSoln} states that the
boundary conditions is necessary. While \cite{JungNumericalSchemeForLiftedFokkerPlanck}
imposes the periodic boundaries, the reasoning does not seem apparent.
We shall show in the example below that, despite $L^{*}(t)$ is $T$-periodic,
a solution to the PDE need not be periodic and relaxing such condition,
perhaps expectedly, one can have infinitely many solutions.
\begin{example}
\label{exa:DensityofOU}The one-dimensional periodically-forced Ornstein-Uhlenbeck
process from Example \ref{exa:PeriodicOUExample} has its Fokker-Planck
operator given explicitly by
\begin{align*}
\mathcal{L}^{*}(t)q & =-\partial_{t}q-\partial_{x}((S(t)-\alpha x)q)+\frac{\sigma^{2}}{2}\partial_{x}^{2}q=-\partial_{t}q-S(t)\partial_{x}q+\alpha q+\alpha x\partial_{x}q+\frac{\sigma^{2}}{2}\partial_{x}^{2}q
\end{align*}
and the periodic measure is $\rho_{t}=\mathcal{N}\left(\xi(t),\frac{\sigma^{2}}{2\alpha}\right)$,
where $\xi(t)=e^{-\alpha t}\int_{-\infty}^{t}e^{\alpha r}S(r)dr$.
Here, the density of the periodic measure is given by 
\[
q(t,x)=\frac{1}{\sqrt{\pi\sigma^{2}/\alpha}}\exp\left(-\frac{(x-\xi(s))^{2}}{\sigma^{2}/\alpha}\right).
\]
We compute 
\[
\dot{\xi}=-\alpha\xi+S(t),\quad\partial_{t}q=2\frac{\alpha}{\sigma^{2}}\dot{\xi}(x-\xi)q,\quad\partial_{x}q=-2\frac{\alpha}{\sigma^{2}}(x-\xi)q,
\]
and
\begin{align*}
\partial_{x}^{2}q & =-2\frac{\alpha}{\sigma^{2}}\left[\partial_{x}(xq)-\xi\partial_{x}q\right]=-2\frac{\alpha}{\sigma^{2}}\left[1-2\frac{\alpha}{\sigma^{2}}(x-\xi)^{2}\right]q.
\end{align*}
Hence, substituting directly,
\begin{align*}
\frac{\mathcal{L}^{*}(t)q}{q} & =-2\frac{\alpha}{\sigma^{2}}\dot{\xi}(x-\xi)+2\frac{\alpha}{\sigma^{2}}S(t)(x-\xi)+\alpha-2\frac{\alpha^{2}}{\sigma^{2}}x(x-\xi)-\alpha\left[1-2\frac{\alpha}{\sigma^{2}}(x-\xi)^{2}\right]\\
 & =2\frac{\alpha^{2}}{\sigma^{2}}\xi(x-\xi)-2\frac{\alpha^{2}}{\sigma^{2}}x(x-\xi)+2\frac{\alpha^{2}}{\sigma^{2}}(x-\xi)^{2}\\
 & =0.
\end{align*}
Thus indeed the $q$ satisfies (\ref{eq:PeriodicFokkerPlanck}). We
show that the periodic condition of (\ref{eq:PeriodicFokkerPlanck})
cannot simply be dropped because of periodic coefficients. From (\ref{eq:LawOfPeriodicMultidimensionalOU}),
the transition density is explicitly given by
\[
p(s,t,y,x)=\frac{1}{\sqrt{\frac{\sigma^{2}}{\alpha}(1-e^{-2\alpha(t-s)})}}\exp\left(-\frac{(x-e^{-\alpha(t-s)}y-J(s,t)}{\frac{\sigma^{2}}{\alpha}(1-e^{-2\alpha(t-s)})}\right),
\]
satisfies $-\partial_{t}p(t,x)+L^{*}(t)p(t,x)=0$ for every fixed
initial time $s$ and point $y$. However, $p$ is not periodic as
$J$ is not periodic. Since there is a non-periodic solution for every
$y\in\mathbb{R}$, there are, in fact, infinite number of solutions
to the PDE if one relaxed the demand of periodicity. 
\end{example}
\textbf{Acknowledgements.} We would like to acknowledge the financial
support of a Royal Society Newton fund grant (ref. NA150344) and an
EPSRC Established Career Fellowship to HZ (ref. EP/S005293/1)

\end{document}